\documentclass[12pt]{amsart}

\usepackage[dvipdfmx]{hyperref}
\usepackage{color}
\usepackage{amsaddr}
\usepackage{comment}
\usepackage{mathtools}

\def\R{\mathbb{R}}
\def\Z{\mathbb{Z}}
\def\N{\mathbb{N}}
\def\calB{\mathcal{B}}

\def\calF{\mathcal{F}}
\def\calH{\mathcal{H}}

\def\calM{\mathcal{M}}
\def\calP{\mathcal{P}}
\def\calS{\mathcal{S}}
\def\calU{\mathcal{U}}
\def\calZ{\mathcal{Z}}

\def\e{\varepsilon}
\def\ttau{\tilde{\tau}}

\def\tc{\tilde{c}}

\def\bv{\overline{v}}

\def\tlambda{\tilde{\lambda}}

\def\tlambda{\tilde{\lambda}}

\theoremstyle{plain}
	\newtheorem{theorem}{Theorem}[section]
	\newtheorem{lemma}[theorem]{Lemma}
	\newtheorem{corollary}[theorem]{Corollary}
	\newtheorem{definition}[theorem]{Definition}
	
	\newtheorem{remark}[theorem]{Remark}

	\newtheorem{example}[theorem]{Example}
\theoremstyle{plain}

\makeatletter
 \@addtoreset{equation}{section}
\makeatother

\begin{document}
\title[Infinite multiplicity]{
Infinite multiplicity of positive solutions of\\
 an inhomogeneous supercritical elliptic equation on $\R^N$
}

\author{Sho Katayama*}
\thanks{*Corresponding author}
\thanks{SK was supported by Grant-in-Aid for JSPS Fellows Grant Number 23KJ0645, and FoPM, WINGS Program, the University of Tokyo.}
\address{Graduate School of Mathematical Sciences, The University of Tokyo,\\
 3-8-1 Komaba, Meguro-ku, Tokyo 153-8914, Japan\\
{\rm{\texttt{katayama-sho572@g.ecc.u-tokyo.ac.jp}}}
}

\author{Yasuhito Miyamoto}
\thanks{ORCiD of YM is 0000-0002-7766-1849}
\thanks{YM was supported by JSPS KAKENHI Grant Number 24K00530.}
\address{Graduate School of Mathematical Sciences, The University of Tokyo,\\
 3-8-1 Komaba, Meguro-ku, Tokyo 153-8914, Japan\\
{\rm{\texttt{miyamoto@ms.u-tokyo.ac.jp}}}
}

\begin{abstract}
We are concerned with positive radial solutions of the inhomogeneous elliptic equation $\Delta u+K(|x|)u^p+\mu f(|x|)=0$ on $\R^N$, where $N\ge 3$, $\mu>0$ and $K$ and $f$ are nonnegative nontrivial functions.
If $K(r)\sim r^{\alpha}$, $\alpha>-2$, near $r=0$, $K(r)\sim r^{\beta}$, $\beta>-2$, near $r=\infty$ and certain assumptions on $f$ are imposed, then the problem has a unique positive radial singular solution for a certain range of $\mu$.
We show that existence of a positive radial singular solution is equivalent to existence of infinitely many positive bounded solutions which are not uniformly bounded, if $p$ is between the critical Sobolev exponent $p_S(\alpha)$ and Joseph-Lundgren exponent $p_{JL}(\alpha)$.
Using these theorems, we establish existence of infinitely many positive bounded solutions which are not uniformly bounded, for $p_S(\alpha)<p<p_{JL}(\alpha)$ if $K(r)=r^{-\alpha}$, $\alpha>-2$.
\end{abstract}

\date{\today}
\maketitle


\section{Introduction and main theorems}
In this paper we consider the inhomogeneous elliptic equation on $\R^N$
\begin{equation}\label{PDE}
\begin{cases}
\Delta u+K(x)u^p+\mu f(x)=0& \textrm{in}\ \ \R^N,\\
u>0 & \textrm{in}\ \ \R^N,
\end{cases}
\end{equation}
where $p>1$, $\mu\ge 0$, $N\ge 3$, $K(x)$ is a positive function and $f(x)$ is a nonnegative nontrivial function.
When $\mu=0$, \eqref{PDE} arises in conformal Riemannian geometry and astrophysics. 
See {\it e.g.}, \cite{C57,LN88}.
On the other hand, when $\mu>0$ and $K\equiv 1$, \eqref{PDE} arises in probability theory, {\it e.g.} the super-Brownian motion \cite{B96,L93}.
The problem \eqref{PDE} is their generalization.
Various sufficient conditions on $K$, $f$ and $p$ for the existence and multiplicities have been studied for three decades.

Throughout the present paper, we assume that $K(x)$ and $f(x)$ are radially symmetric and study radial solutions of \eqref{PDE}.
Then, \eqref{PDE} can be reduced to the ODE
\begin{equation}\label{ODE}
\begin{cases}
u''+\frac{N-1}{r}u'+K(r)u^p+\mu f(r)=0 &\text{for}\ r>0,\\
u>0 & \text{for}\ r>0.
\end{cases}
\end{equation}
We assume that $K(r)$ is a positive continuous function such that the following \eqref{K0} and \eqref{Kinfty} hold:
\begin{equation}\label{K0}
K(r)=(k_0+o(1))r^\alpha\ \textrm{as}\ r\to 0\ \textrm{for some $\alpha>-2$ and $k_0>0$},
\end{equation}
\begin{equation}\label{Kinfty}
K(r)=(k_{\infty}+o(1))r^\beta\ \textrm{as}\ r\to\infty\ \textrm{for some $\beta>-2$ and $k_{\infty}>0$}.
\end{equation}
A typical example is $K(r)=r^{\alpha}$, $\alpha>-2$.
$f(r)$ is a nonnegative nontrivial continuous function such that
\begin{equation}\label{f02}
f(r)=O(r^{\nu})\ \text{as}\ r\to 0\ \ \text{for some}\ \nu>-2,
\end{equation}
\begin{equation}\label{finfty2}
f(r)=O(r^{-q}) \ \text{as}\ \ r\to\infty\ \ \text{for some}\ q>N.
\end{equation}

There are two important exponents in the study of \eqref{ODE}:
\begin{align*}
p_S(\alpha)&:=\frac{N+2+2\alpha}{N-2},\\
p_{JL}(\alpha)&:=
\begin{cases}
1+\frac{2(2+\alpha)}{N-4-\alpha-\sqrt{(2+\alpha)(2N-2+\alpha)}} & \text{for}\ N>10+4\alpha,\\
\infty & \text{for}\ N\le 10+4\alpha.
\end{cases}
\end{align*}
$p_S(\alpha)$ is the critical Sobolev exponent and $p_{JL}(\alpha)$ is the so-called Joseph-Lundgren exponent.
\begin{definition}\label{S1D1}
We call $u(r)\in C^2(0,\infty)$ a bounded solution of \eqref{ODE} if $u$ satisfies \eqref{ODE} pointwise and there exists a limit
$$
\lim_{r\to 0}u(r)<\infty.
$$
We call $u^*(r)\in C^2(0,\infty)$ a singular solution of \eqref{ODE} if $u^*$ satisfies \eqref{ODE} pointwise and
$$
\lim_{r\to 0}u(r)=\infty.
$$
\end{definition}
We prove existence of infinitely many solutions of \eqref{ODE}.
In the previous studies, the comparison principle or a functional analytic method were used.
However, in this paper we use a different approach, where a singular solution of \eqref{ODE} plays a crucial role.

The first result is a sufficient condition for existence of a singular solution.
\begin{theorem}\label{S1T1}
Suppose that $N\ge 3$, \eqref{K0}, and \eqref{f02} hold.
Assume that $p>p_S(\alpha)$.
If the equation \eqref{ODE} has infinitely many bounded solutions $\{u_j\}_{j=0}^{\infty}$ such that $\left\|u_j\right\|_{L^{\infty}}\to\infty$ as $j\to\infty$, then \eqref{ODE} has a singular solution $u^*(r)$.
\end{theorem}
On the other hand, under the assumption $p_S(\alpha)<p<p_{JL}(\alpha)$, a sufficient condition for existence of infinitely many bounded solutions of \eqref{ODE} is given as follows:
\begin{theorem}\label{S1T2}
Suppose that $N\ge 3$, \eqref{K0}, and \eqref{f02} hold.
Assume that $p_S(\alpha)<p<p_{JL}(\alpha)$.
If \eqref{ODE} has a singular solution $u^*(r)$, then \eqref{ODE} has infinitely many bounded solutions $\{u_j\}_{j=0}^{\infty}$ such that $\left\|u_j\right\|_{\infty}\to\infty$ as $j\to\infty$.
\end{theorem}
The following corollary immediately follows from Theorems~\ref{S1T1} and \ref{S1T2}:
\begin{corollary}\label{S1C1}
Suppose that $N\ge 3$, \eqref{K0}, and \eqref{f02} hold.
Assume that $p_S(\alpha)<p<p_{JL}(\alpha)$.
Then, the following (i) and (ii) are equivalent:\\
(i) \eqref{ODE} has a singular solution $u^*(r)$.\\
(ii) A set of bounded solutions of \eqref{ODE} is not uniformly bounded.
In particular, \eqref{ODE} has infinitely many bounded solutions.
\end{corollary}
Because of Corollary~\ref{S1C1}, we are interested in existence of singular solutions of \eqref{ODE}.
We assume \eqref{Kinfty} and \eqref{finfty2} in addition to \eqref{K0} and \eqref{f02}.
We need some notation.
Let $u$ be a solution of \eqref{ODE}.
Then, $u$ is said a fast-decay solution if
$$
\limsup_{r\to\infty}r^{N-2}u(r)<\infty,
$$
and $u$ is said a slow-decay solution if
$$
\limsup_{r\to\infty}r^{N-2}u(r)=\infty,
$$
The main result of the paper is the following:
\begin{theorem}\label{S1T3}
Suppose that $N\ge 3$ and \eqref{K0}--\eqref{finfty2} hold.
Assume that $p>\max\{p_S(\alpha),p_S(\beta)\}$.
Then one of the following holds:
\begin{enumerate}
\item \eqref{ODE} has no singular solutions for all $\mu\ge 0$.
\item There are finitely many numbers $\{\mu_1,\mu_2,\ldots,\mu_{N_f}\}$, $N_f\ge 1$, such that $0\le\mu_1<\ldots<\mu_{N_f}<\infty$ and the following hold:
\begin{enumerate}
  \item If $\mu\in\{\mu_1,\ldots,\mu_{N_f}\}$, then \eqref{ODE} has a singular fast-decay solution.
  \item For each $i\in\{1,\ldots,N_f-1\}$, one of the following holds.
  \begin{enumerate}
    \item For all $\mu\in(\mu_i,\mu_{i+1})$, \eqref{ODE} has a singular slow-decay solution.
    \item For all $\mu\in(\mu_i,\mu_{i+1})$, \eqref{ODE} has no singular solutions.
  \end{enumerate}
  Furthermore, the same dichotomy holds for the interval $[0,\mu_1)$ if $\mu_1>0$.
  \item If $\mu>\mu_{N_f}$, then \eqref{ODE} has no singular solutions.
\end{enumerate}
\end{enumerate}
\end{theorem}
\begin{example}
(i) Let $N_f=1$ and $\mu_1=0$.
Then \eqref{ODE} has a fast-decay singular solution for $\mu=0$ and \eqref{ODE} has no singular solution for $\mu>0$.\\
(ii) Let $N_f=1$ and $\mu_1>0$.
Then for $0\le\mu<\mu_1$, \eqref{ODE} has a slow-decay singular solution or no singular solution.
For $\mu=\mu_1$, \eqref{ODE} has a fast-decay singular solution.
For $\mu>\mu_1$, \eqref{ODE} has no singular solution. 
\end{example}

In the case $K(r)\equiv r^\alpha$, $\alpha>-2$, the problem \eqref{ODE} with $\mu=0$ has a singular slow-decay solution
\[
u^*_0(r)=\left\{\frac{2+\alpha}{p-1}\left(N-2-\frac{2+\alpha}{p-1}\right)\right\}^{\frac{1}{p-1}}r^{-\frac{2+\alpha}{p-1}}.
\]
This together with Theorem~\ref{S1T3} derives the following corollary:
\begin{corollary}\label{S1T3'}
Let $K(r)\equiv r^{\alpha}$ and $p>p_S(\alpha)$. Assume \eqref{f02} and \eqref{finfty2}.
There exists $\mu_1>0$ such that for $\mu\in[0,\mu_1)$, \eqref{ODE} has a singular slow-decay solution.
\end{corollary}
\begin{proof}
When $\mu=0$, \eqref{ODE} has a singular slow-decay solution $u_0^*$.
On the other hand, \eqref{ODE} does not have a singular solution for large $\mu>0$, because of Theorem~\ref{S1T3}~(2)(c).
It follows from Theorem~\ref{S1T3}~(2)(b) that there is $\mu_1>0$ such that Theorem~\ref{S1T3}~(2)(b)(i) occurs for $\mu\in [0,\mu_1)$.
Thus, \eqref{ODE} has a singular slow-decay solution for $\mu\in [0,\mu_1)$.
\end{proof}
The following is an immediate consequence of Theorems~\ref{S1T2} and Corollary~\ref{S1T3'}.
\begin{corollary}\label{S1C6}
Let $K(r)\equiv r^\alpha$ and $p_S(\alpha)<p<p_{JL}(\alpha)$. Assume \eqref{f02} and \eqref{finfty2}.
Let $\mu_1$ be given in Corollary~\ref{S1T3'}.
Then for each $\mu\in [0,\mu_1)$, \eqref{ODE} has infinitely many bounded solutions that are not uniformly bounded.
\end{corollary}
Corollary~\ref{S1C6} is new even in the case $K(r)\equiv 1$.
When $p_S(\alpha)<p<p_{JL}(\alpha)$, existence of a singular slow-decay solution of \eqref{ODE} for $\mu=0$ leads to existence of infinitely many bounded solutions of \eqref{ODE} for small $\mu>0$ even in the case where $K$ is general.
\bigskip

Let us recall previous studies.
When $\mu=0$, there is a vast amount of literature, and existence, nonexistence and multiplicities of solutions and other properties have been studied in these four decades.
See {\it e.g.,} \cite{DN85,LN88} and references therein.
However, the case $\mu>0$ has not been studied as extensively as the case $\mu=0$.
Bernard~\cite{B96} first studied existence of a solution of \eqref{PDE} with $K(x)\equiv 1$ when $\mu>0$.
In \cite{B96} it was shown that if $N\ge 3$, $p>N/(N-2)$, $f\in C^{0,\gamma}(\R^N)$ for $0<\gamma\le 1$ and
$$
0\le \mu f(x)\le\frac{1}{(p-1)^{1/(p-1)}}\left\{\frac{2}{p}\left(N-2-\frac{2}{p-1}\right)\right\}^{\frac{p}{p-1}}\frac{1}{(1+|x|^2)^{p/(p-1)}},\ \ f\not\equiv 0,
$$
then \eqref{PDE} with $K(x)\equiv 1$ has a solution.
After \cite{B96} this research has been generalized in various directions.
In \cite{B02,B09,BCP02,BN01,DGL08,DLY09,DY08,HMP06,IK24,LG11} existence of solutions of \eqref{PDE} was studied under various conditions on $K$ and $f$.
In \cite{B02,B09,DGL08,DLY09,DY08} asymptotic behaviors of solutions near $|x|=\infty$ were also studied.
Other qualitative properties were studied in \cite{DLY09,IK24}.

Existence of infinitely many solutions of \eqref{PDE} in the case $\mu>0$ was studied in \cite{BCP02,BN01,LG11}.
In \cite{BN01} Bae and Ni proved that there is $\mu^*>0$ such that \eqref{PDE} with $K(x)\equiv 1$ has infinitely many solutions for every $0<\mu<\mu^*$ if $p>p_{JL}(0)$ and
$$
\max\{\pm f(x),0\}\le |x|^{-q_{\pm}}\ \ \text{for $x$ near $\infty$}.
$$
Here, $q_+>N-\lambda_2$, $q_->N-\lambda_2-\frac{2}{p-1}$ and
$$
\lambda_2:=\frac{1}{2}\left\{
N-2-\frac{4}{p-1}+\sqrt{\left(N-2-\frac{4}{p-1}\right)^2-8\left(N-2-\frac{2}{p-1}\right)}\right\}.
$$
Moreover, in the case $p=p_{JL}(0)$ infinitely many solutions exist if $f$ has a compact support or $f$ does not change sign.

This theorem was generalized by Bae-Chang-Pahk~\cite{BCP02}.
In \cite{BCP02} they proved the following:
\begin{theorem}\label{BCP}
If $p\ge p_{JL}(\alpha)$, $\alpha>-2$, $K$ is a locally H\"{o}lder continuous function in $\R^N\setminus\{0\}$, $K(x)\ge 0$,
\begin{align*}
K(x)&=O(|x|^{\alpha})\ \text{as}\ |x|\to 0\ \text{for some}\ \alpha>-2,\\
K(x)&=c|x|^{\beta}+O(|x|^{-d})\ \text{as}\ |x|\to \infty
\end{align*}
for some $c>0$ and $d>N-\lambda_2(\alpha)-\frac{(2+\alpha)(p+1)}{p-1}$, and $f$ satisfies
\begin{align}
&f(x)=O(|x|^{\nu})\ \text{as}\ |x|\to 0\ \text{for some}\ \nu>-2,\label{BCPE1}\\
&-(1+|x|^{(2+\alpha)p/(p-1)})f(x)\le\min_{|z|=|x|}K(z),\label{BCPE2}\\
&\text{near}\ |x|=\infty,\ f(x)=O(|x|^{-q})\ \text{for some}\ q>N-\lambda_2(\alpha)-\frac{2+\alpha}{p-1},\label{BCPE3}
\end{align}
then there is $\mu^*>0$ such that for every $0\le\mu <\mu^*$, \eqref{PDE} has infinitely many solutions.
Here,
$$
\lambda_2(\alpha):=\frac{1}{2}\left\{
N-2-\frac{2(2+\alpha)}{p-1}+\sqrt{\left(N-2-\frac{2(2+\alpha)}{p-1}\right)^2-4(\alpha+2)\left(N-2-\frac{2+\alpha}{p-1}\right)}\right\}.
$$
\end{theorem}
In \cite{LG11} Lai and Ge proved the following:
\begin{theorem}\label{LG}
Let $K(x)\equiv 1$ and $\mu=1$.
If $N\ge 4$, $p>(N+1)/(N-3)$, $f$ satisfies
\begin{align*}
&f(x)=\eta(x)f_1(x),\ f\not\equiv 0,\\
&f_1(x)\in C^{0,\gamma}(\R^N),\ 0\le  f_1(x)<|x|^{-q}\ \text{for some}\ q>2+\frac{2}{p-1},\\
&\eta(x)\in C^{\infty}(\R^N),\ 0\le\eta(x)\le 1,
 \eta(x)=
\begin{cases}
0 & \text{for}\ |x|\le R_1,\\
1 & \text{for}\ |x|\ge R_1+1,
\end{cases}
\end{align*}
then there exists a large $R_1>0$ such that \eqref{PDE} has a continuum of solutions.
If $(N+2)/(N-2)<p\le (N+1)/(N-3)$, the same result holds provided that $f$ is symmetric with respect to $n$ coordinate axes.
\end{theorem}
Let us consider the case $K(r)\equiv r^{\alpha}$, $\alpha>-2$.
Although our theorems hold only for radial solutions, Corollary~\ref{S1C6} covers the range $p_S(\alpha)<p<p_{JL}(\alpha)$.
Therefore, when $K(r)\equiv r^{\alpha}$, Corollary~\ref{S1C6} and Theorem~\ref{BCP} cover the whole supercritical range $p>p_S(\alpha)$ when $f$ satisfies the assumptions of both Corollary~\ref{S1C6} and Theorem~\ref{BCP}.

Next, we compare our theorems with Theorem~\ref{LG}.
In the case where $f$ is symmetric with respect to $n$ coordinate axes, Theorem~\ref{LG} guarantees existence of infinitely many solutions of \eqref{PDE} for all $p>p_S(0)$, while Corollary~\ref{S1C6} holds for $p_S(\alpha)<p<p_{JL}(\alpha)$.
In Theorem~\ref{LG} $f$ needs to be identically zero near the origin, while our assumption on $f$ admits a singularity as long as \eqref{f02} is satisfied.\\

Our idea of the proof of Theorem~\ref{S1T1} is obtaining a singular solution $u^*$ as a pointwise limit of bounded solutions $\{u_j\}_{j=0}^\infty$ with $\|u_j\|_{L^\infty}\to\infty$ as $j\to\infty$ (see Theorem~\ref{convtosing} and Section~5).
We do not employ the Pohozaev identity so that conditions on $K(r)$ are not assumed except those on behaviors near $0$ or $\infty$.
Instead, we compare the regular solution $u(r,\zeta)$ of problem~\eqref{ODE} with the explicit singular solution $u_0^*$ to
the homogeneous problem
\[
u''+\frac{N-1}{r}u'+r^\alpha u^p=0\ \textrm{for}\ r>0,\ u(r)>0.
\]
We first prove that asymptotic behavior of singular solutions to problem~\eqref{ODE} as $r\to+0$ is the same as that of $u^*_0$ (see Lemma~\ref{u0asypro}), by imposing the Emden--Fowler transform (see \eqref{wg},\eqref{weq}) and applying a lemma on asymptotic behavior of solutions to perturbed ODEs (see Lemma~\ref{S2L2} and \cite[Lemma 3.2]{MN20}). Furthermore, we estimate differences of regular solutions $u(r,\zeta)$ with $u(0,\zeta)=\zeta$ to problem~\eqref{ODE}, from the regular solution $\overline{u}(r,\zeta)$ with $\overline{u}(0,\zeta)=\zeta$ to the homogeneous problem near $0$, by the Gronwall inequality (see Lemma~\ref{uapprox1}). Finally, we compare $u(r,\zeta)$ with $u_0^*$ by applying a perturbation argument on the Emden--Fowler--transformed equation (see Lemma~\ref{uapprox2}).

In the proof of Theorem~\ref{S1T2}, it is crucial to count the intersection number of $u(r,\zeta)$ and $u^*$ near $0$.
By comparisons mentioned above, we observe that the intersection number $\calZ_{(0,\rho)}[u^*-u(\cdot,\zeta)]$ of $u(r,\zeta)$ and $u^*$ in $(0,\rho)$ is similar to $\calZ_{(0,\rho)}[u_0^*-\overline{u}(\cdot,\zeta)]$ if $\zeta$ is very large. In particular, if $p<p_{JL}(\alpha)$, since $\calZ_{(0,\rho)}[u_0^*-\overline{u}(\cdot,\zeta)]\to\infty$ as $\zeta\to\infty$ (see \ref{ubarintersect} and \cite{MT17,W93}), $\calZ_{(0,\rho)}[u_0-u(\cdot,\zeta)]\to\infty$ also holds (see Theorem~\ref{inftyintersect}). Because $u(r,\zeta)$ must be an entire positive solution for $\zeta$ at which $\calZ_{(0,\infty)}[u_0-u(\cdot,\zeta)]$ increases, this implies the conclusion of Theorem~\ref{S1T2}.

The idea of the proof of Theorem~\ref{S1T3} is comparing the values of the derivatives at a certain point $R_1\gg 0$ of singular solutions to those of fast-decay solutions to the problem
\begin{equation}\label{farequation}
v''+\frac{N-1}{r}v'+K(r)v^p+\mu f(r)=0\ \textrm{for}\ r>R_1.
\end{equation}
We first establish the analytic dependence on $\mu$ of the singular solutions to
\[
    u''+\frac{N-1}{r}u'+K(r)\max\{u,0\}^p+\mu f(r)=0\ \textrm{for $r>0$},\ \lim_{r\to+0}r^{\theta} u(r)=\gamma,
\]
by the implicit function theorem (see Theorem~\ref{singexists}). We next parametrize fast-decay solutions $v=v_{\eta,\mu}$ to problem~\eqref{farequation} by the values of $\eta:=\lim\limits_{r\to\infty}r^{N-2}v(r)$ and $\mu$ analytically, and show that $\eta$ can be expressed by an analytic function $\calH(\xi,\mu)$ of $\xi=v_{\eta,\mu}(R_1)$ and $\mu$. This enables us to characterize the values of $\mu$ which admit singular fast-decay solutions to problem~\eqref{ODE} by the null points of the analytic function $u^{*\prime}_\mu(R_1)-v'_{\calH(u^*_{\mu}(R_1),\mu),\mu}(R_1)$ (see Lemma~\ref{SFsolconstraint}). By the identity theorem, such $\mu$ only exists discretely, which together with the upper bound of $\mu$ (see Lemma~\ref{mubound}) implies the finiteness (see Lemma~\ref{SFfinite}). Finally, we see that both the set of the value of $\mu$ which admits singular slow-decay solutions and  the set of the value of $\mu$ which admits no singular solutions are open in $[0,\infty)$, and complete the proof of Theorem~\ref{S1T3}.\\

This paper consists of six sections.
In Section~2 we establish various apriori estimates for both bounded and singular solutions of \eqref{ODE}.
In particular, the principal term of a singular solutions is uniquely determined by Lemma~\ref{u0asypro}.
In Section~3 we construct a singular solution of \eqref{ODE} near $r=0$ and  prove the uniqueness of a singular solution.
Note that the uniqueness is proved from information only near $0$.
Moreover, we show that a large bounded solution converges to the singular solution in $C^2_{\rm loc}(0,r^*)$ for some $r^*>0$.
In Section~4 we show that the intersection number between a large bounded solution $u(r,\zeta)$ and the singular solution $u^*(r)$ near $r=0$ diverges as $\zeta\to\infty$.
In Section~5 we prove Theorems~\ref{S1T1} and \ref{S1T2}, using the property of the intersection number proved in Section~4.
In Section~6 we prove Theorem~\ref{S1T3}.
In particular we show that the set of $\mu$ that has a slow-decay solution is open and that the set of $\mu$ that has no positive solution is also open.
Moreover, we show that the set of $\mu$ that has a fast-decay solution is finite, using an analyticity argument.
Then, the conclusions of Theorem~\ref{S1T3} follow from the structure of the three sets of $\mu$.
As a consequence, Corollaries~\ref{S1T3'} and \ref{S1C6} are established.

\section{Apriori estimates}
The goal of this section is to prove Lemma~\ref{u0asypro}.
The proofs in this section are based on \cite[Lemmas 2.1 and 2.2]{MN20} and \cite[Lemma 3.2]{MN20}.

We define
$$
  \theta:=\frac{2+\alpha}{p-1},\ a:=N-2-2\theta,\ c:=N-2-\theta,\ A^{p-1}:=\theta c,
$$
and
\begin{equation}\label{wg}
  w(t):=e^{\theta t}u(e^{t}),\ L(t):=e^{-\alpha t}K(e^{t}),\ g(t):=e^{(2+\theta) t}f(e^t),
\end{equation}
so that equation \eqref{ODE} is equivalent to
\begin{equation}\label{weq}
  w''+aw'-A^{p-1}w+L(t)w^p+\mu g(t)=0\ \textrm{for $t\in\R$}.
\end{equation}
We note that \eqref{K0} implies that
\begin{equation}\label{L0}
  L(t)=k_0+o(1)\ \textrm{as $t\to-\infty$},
\end{equation}
and that
$$
a>0\ \textrm{if $p>p_S(\alpha)$.}
$$
Since \eqref{f02} and \eqref{finfty2} are assumed, $f$ satisfies the following:
\begin{gather}
\int_0^1 rf(r)\,dr<\infty,\label{f0}\\
\int_1^\infty r^{N-1}f(r)\,dr<\infty.\label{finfty}
\end{gather}
It follows from \eqref{f02} and \eqref{f0} that
\begin{equation}\label{g0}
  \lim_{t\to-\infty}g(t)=0,\quad \int_{-\infty}^0 e^{-\theta s}g(s)\,ds<\infty.
\end{equation}

We mention existence of a solution of the initial value problem
\begin{equation}\label{S2E0}
\begin{cases}
u''+\frac{N-1}{r}u'+K(r)\max\{u,0\}^p+\mu f(r)=0 & \text{for}\ r>0,\\
u(0)=\zeta,
\end{cases}
\end{equation}
since $f(r)$ may have a singularity at $r=0$.
By applying the contraction mapping theorem to the problem
$$
u(r)=\zeta-\int_0^r\frac{s^{2-N}-r^{2-N}}{N-2}s^{N-1}
\left(K(s)u(s)^p+\mu f(s)\right)ds,
$$
we can prove that a solution of \eqref{S2E0} near $r=0$ in the sense of Definition~\ref{S1D1} uniquely exists.
Note that when $u(r)$ is bounded near $r=0$, the RHS is integrable, since $K(r)=O(r^{\alpha})$, $\alpha>-2$, (by \eqref{K0}) and $f(s)=O(s^{\gamma})$, $\gamma>-2$ (by \eqref{f02}).
Therefore, the solution of \eqref{S2E0} near $r=0$ exists for all $\zeta>0$.
\begin{lemma}\label{S2L1}
Let $N\ge 3$ and $p>1$.
Assume that \eqref{K0} and \eqref{f0} hold.\\
(i) For both bounded and singular positive solutions $u$ of \eqref{ODE} near $r=0$, there are constants $r_0\in(0,\infty]$ and $C_0>0$ such that
\begin{equation}\label{asy1}
  \sup_{r\in(0,r_0)}r^\theta u(r)\le C_0.
\end{equation}
(ii) If $u$ is singular, then
\begin{equation}\label{asy2}
  \limsup_{r\to 0}r^{\theta}u(r)>0.
\end{equation}
\end{lemma}

The following proof is a modification of \cite[Lemmas 2.1 and 2.2]{MN20} for homogeneous equations.
We present a proof for readers' convenience.
\begin{proof}
(i) It is clear that $u(r)$ is decreasing near $r=0$ for both bounded and singular positive solutions.
If $u$ has the first critical point $r_1\in (0,\infty)$, then $u''(r_1)=-K(r_1)u(r_1)^p-\mu f(r_1)<0$, and hence $r_1$ is a local maximum point.
Then, it contradicts that $u$ is decreasing near $r=0$.
Therefore, $u$ has no critical point, and hence $u'(r)<0$ for $r>0$.

Let $r_1>0$ be arbitrarily small.
Then, we have
\[
-r^{N-1}u'(r)\ge -r_1^{N-1}u'(r_1)+\int_{r_1}^r s^{N-1}K(s)u(s)^pds.
\]
Since $r_1>0$ is arbitrary small, we obtain
\[
-r^{N-1}u'(r)\ge\int_0^rs^{N-1}K(s)u(s)^pds.
\]
Since $K(r)\ge k_0r^{\alpha}/2$ for small $r>0$ and $u(r)$ is decreasing, we have
\[
-r^{N-1}u'(r)\ge\frac{k_0}{2}u(r)^p\int_0^rs^{N-1+\alpha}ds.
\]
It implies that $-u'(r)/u(r)^p\ge Cr^{1+\alpha}$ for small $r>0$.
Integrating it over $(0,r)$, we deduce that $u(r)^{1-p}\ge Cr^{2+\alpha}$ for small $r>0$.
Thus,
\[
u(r)^{p-1}\le Cr^{-(2+\alpha)}
\]
for small $r>0$. We obtain \eqref{asy1}.\\  
(ii) Let $u$ be a singular solution.
We show that
\begin{equation}\label{S2L1E1}
\liminf_{r\to 0}\left(-r^{N-1}u'(r)\right)=0.
\end{equation}
Assume on the contrary that $\liminf_{r\to 0}\left(-r^{N-1}u'(r)\right)=c>0$.
Then, there exists $r_1>0$ such that $-r^{N-1}u'(r)\ge c/2$ for $0<r<r_1$.
Integrating it over $(r,r_1)$, we obtain
\[
u(r)\ge u(r_1)+\frac{c}{2(N-2)}(r^{2-N}-r_1^{2-N}).
\]
This contradicts \eqref{asy1}.
Thus, \eqref{S2L1E1} holds.
  
Because of \eqref{S2L1E1}, there exists $\{r_k\}$ such that $r_k\to 0$ and $r_k^{N-1}u'(r_k)\to 0$ as $k\to\infty$.
Since $-r^{N-1}u'(r)+r_k^{N-1}u'(r_k)=\int_{r_k}^rs^{N-1}(K(s)u(s)^p+\mu f(s))ds$, by letting $k\to\infty$ we have
\begin{equation}\label{S2L1E2}
-r^{N-1}u'(r)=\int_0^rs^{N-1}(K(s)u(s)^p+\mu f(s))ds.
\end{equation}

Next, we prove \eqref{asy2}.
Assume the contrary, {\it i.e.},
\begin{equation}\label{S2L1E3}
\lim_{r\to 0}r^{\theta}u(r)=0.
\end{equation}
Then, it is equivalent to $w(t)\to 0$ as $t\to -\infty$.
We show that there is a small $r_1>0$ such that
\begin{equation}\label{S2L1E4}
(r^{\theta}u(r))'\ge 0\ \ \textrm{for}\ \ 0<r<r_1.
\end{equation}
By \eqref{weq} we see that there is $t_1\in\R$ such that
\begin{align*}
\left(e^{at}w'(t)\right)' & = e^{at}\left( A^{p-1}-L(t)w(t)^{p-1}-\frac{\mu e^{2t}f(e^t)}{u(e^t)}\right)w\\
&\ge e^{at}\left(A^{p-1}-\frac{A^{p-1}}{4}-\frac{A^{p-1}}{4}\right)w>0
\end{align*}
for $t<t_1$.
Therefore, $e^{at}w'(t)$ is increasing for $t<t_1$.
We show that
\begin{equation}\label{S2L1E5}
w'(t)\ge 0\ \ \textrm{for}\ \ t<t_1.
\end{equation}
Assume the contrary, {\it i.e.,} there exists $t_2\in (-\infty,t_1]$ such that $w'(t_2)<0$.
Then,
\[
e^{at}w'(t)\le e^{at_2}w'(t_2)<0\ \ \textrm{for}\ \ t<t_2,
\]
and hence $w'(t)\le e^{-at}e^{at_2}w'(t_2)<0$ for $t<t_2$.
Then, $w'(t)\to -\infty$ as $t\to -\infty$.
This is a contradiction, because $w(t)\to 0$ as $t\to -\infty$.
Thus, we obtain \eqref{S2L1E5} which implies \eqref{S2L1E4} with small $r_1>0$.

By \eqref{S2L1E3} we have
\[
r^2\frac{K(r)u(r)^p}{u(r)}=\frac{K(r)}{r^{\alpha}}\left(r^{\theta}u(r)\right)^{p-1}\to 0
\ \ \textrm{as}\ \ r\to 0.
\]
For each small $\e>0$, there exists $r_1>0$ such that
\begin{align}
\int_0^rs^{N-1}\left(K(s)u(s)^p+\mu f(s)\right)ds
&\le (N-2-\theta)\e\int_0^rs^{N-3}u(s)ds+C\int_0^rs^{N-1+\gamma}ds\nonumber\\
&\le (N-2-\theta)\e r^{\theta}u(r)\int_0^rs^{N-3-\theta}ds+Cr^{N+\gamma}\nonumber\\
&\le \e r^{N-2}u(r)+Cr^{N+\gamma}\ \ \textrm{for}\ \ 0<r<r_1,\label{S2L1E6}
\end{align}
where \eqref{f02} and \eqref{S2L1E4} are used.
We can choose $r_2\in (0,r_1)$ such that 
\begin{equation}\label{S2L1E7}
\e r^{N-2}u(r)+Cr^{N+\gamma}=\e r^{N-2}u(r)+Cr^{2+\gamma}r^{N-2}
\le \e r^{N-2}u(r)+\e r^{N-2}
\end{equation}
for $0<r<r_2$.
By \eqref{S2L1E2}, \eqref{S2L1E6} and \eqref{S2L1E7} we have
\[
-r^{N-1}u'(r)\le \e r^{N-2}u(r)+\e r^{N-2}\ \ \text{for}\ \ 0<r<r_2.
\]
This implies that $u(r)r^{\e}\le u(r_2)r_2^{\e}+(r_2^{\e}-r^{\e})$ for $0<r<r_2$.
Thus, $u(r)=O(r^{-\e})$ as $r\to 0$.
Then,
\[
-r^{N-1}u'(r)=\int_0^rs^{N-1}\left(K(s)u(s)^p+\mu f(s)\right)ds
\le C\int_0^rs^{N-1-\e p+\alpha}ds+Cr^{N+\gamma}\le Cr^{N-\e p+\alpha}+Cr^{N+\gamma}.
\]
Since $-u'(r)\le Cr^{1+\alpha-\e p}+Cr^{1+\gamma}$ for $0<r<r_2$ and $\e>0$ can be arbitrarily small, we see that $\limsup_{r\to 0}u(r)<\infty$.
This is a contradiction.
Thus, we obtain $\limsup_{r\to 0}r^{\theta}u(r)>0$.
\end{proof}

\begin{lemma}\label{S2L2}
Assume the same conditions as in Theorem~\ref{S1T1}. Let $w$ be a solution to \eqref{weq} satisfying
\begin{equation}\label{S2L2E1}
0<\limsup_{t\to-\infty}w(t)<\infty.
\end{equation}
Then 
\begin{equation}\label{w+asymp}
\lim\limits_{t\to-\infty}w(t)=\gamma:=k_0^{-\frac{1}{p-1}}A.
\end{equation}
\end{lemma}
\begin{proof}
Set
\[
\psi(t):=-\int_{-\infty}^t \frac{e^{\theta(t-s)}-e^{-c(t-s)}}{N-2}g(s)\,ds,\quad w_1(t):=w(t)-\mu\psi(t),
\]
for $t\in\R$. Since $\theta>0>-c$ and
\[
\int_{-\infty}^t e^{-\theta s}g(s)\,ds=\int_0^{e^t}rf(r)\,dr<\infty
\]
by \eqref{f0}, the function $\psi$ is well-defined and it satisfies
\begin{equation}\label{psiorder}
\psi(t)=o(e^{\theta t})\ \textrm{as}\ t\to-\infty.
\end{equation}
Furthermore, since
\[
\psi''+a\psi'-A^{p-1}\psi=-g
\]
by direct calculation, we see that
\[
w_1''+aw_1'-A^{p-1}w_1+L(t)(w_1+\mu\psi)^p=0.
\]
It is clear that
\[
(-A^{p-1}v+k_0v^p)(v-\gamma)>0\ \textrm{for all $v\in(0,\infty)\setminus\{\gamma\}$}.
\]
It also follows from \eqref{L0}, \eqref{g0}, and \eqref{psiorder} that
\[
|-k_0w_1^p+L(t)(w_1+\mu\psi)^p|\le |L(t)-k_0|w_1^p+L(t)|(w_1+\mu\psi)^p-w_1^p|\to 0\ \textrm{as $t\to-\infty$.}
\]
Using \cite[Lemma 3.2]{MN20} with
\[
H(w)=-A^{p-1}w+k_0w^p,\ G(t,w)=-k_0w^p+L(t)(w+\mu\psi)^p,
\]
we obtain $\lim\limits_{t\to-\infty}w_1(t)=\gamma$. Combining this with \eqref{psiorder} yields \eqref{w+asymp}.
\end{proof}

\begin{lemma}\label{u0asypro}
Assume the same conditions as in Theorem~\ref{S1T1}.
Let $u$ be a singular solution of \eqref{ODE} for $0<r<r_0$ with some $r_0>0$. Then
\begin{equation}\label{S2L3E0}
u(r)=\gamma r^{-\theta}(1+o(1))\ \ \textrm{as}\ \ r\to 0.
\end{equation}
\end{lemma}
\begin{proof}
Since $u$ is a singular solution, by Lemma~\ref{S2L1} we see that $u$ satisfies \eqref{asy1} and \eqref{asy2}.
Let $w(t)$ be defined by \eqref{wg}.
Then $w(t)$ satisfies \eqref{S2L2E1}.
It follows from Lemma~\ref{S2L2} that \eqref{w+asymp} holds, and hence \eqref{S2L3E0} holds.
\end{proof}

\section{Singular solution}
We first start with the following decay estimate for a second order linear ordinary differential equation with perturbed coefficients.
\begin{lemma}\label{pertODE}
Suppose that $a(t)$ and $b(t)$ are continuous functions satisfying $|a(t)-a|<\varepsilon$ and $|b(t)-b|<\varepsilon$ on $(t_0,t_1)$ for some $a,b>0$ and $\varepsilon>0$ {and that} $g$ is a bounded continuous function on $(t_0,t_1)$. Let $z(t)$ be a solution to $z''+a(t)z'+b(t)z=g$. Then
  \begin{equation}\label{pertODEdecayest}
    |z(t)|+|z'(t)|\le C{e^{-(\lambda-2\e)(t-t_0)}}\left(|z(\tau)|+|z'(\tau)|\right)+C\int_{\tau}^t {e^{-(\lambda-2\e)(t-s)}}|g(s)|\,ds
  \end{equation}
for all $t_0<\tau<t<t_1$, where $\lambda>0$ and $C$ depend only on $a,b$. Furthermore, if $2\varepsilon<\lambda$, $t_0=-\infty$, and $\lim\limits_{t\to-\infty}z(t)=0$, then
 \begin{equation}\label{pertODEdecayest'}
   |z(t)|+|z'(t)|\le C\int_{-\infty}^t {e^{-(\lambda-2\e)(t-s)}}|g(s)|\,ds
 \end{equation}
 for all $t\in(-\infty,t_1)$.
\end{lemma}
\begin{proof}
  We first prove \eqref{pertODEdecayest} in the case $g\equiv 0$. Define
  \[
  A:=\begin{pmatrix}
    0 & 1\\
    -b & -a
  \end{pmatrix}.
  \]
Since $a,b>0$, we can easily check that two eigenvalues are negative or real parts of them are negative if they are not real.
There exists $\lambda>0$ such that
  \begin{equation}\label{opnorm}
  |e^{tA}v|\le e^{-t\lambda}|v|\ \ \textrm{for all $t>0$ and $v\in\R^2$}.
  \end{equation}
Since, for any bounded continuous function $h$ on $(t_0,t_1)$, a solution to $Z''+aZ'+bZ=h$ satisfies
  \[
  \begin{pmatrix}
    Z(t)\\
    Z'(t)
  \end{pmatrix}
  ={e^{(t-\tau)A}}\begin{pmatrix}
    {Z(\tau)}\\
    {Z'(\tau)}
    \end{pmatrix}
    +\int_{\tau}^t e^{(t-s)A}
    \begin{pmatrix}
      0\\
      h(s)
    \end{pmatrix}\,ds
  \]
for $t_0<\tau<t<t_1$, a solution to $z''+a(t)z'+b(t)z=g(=0)$ satisfies
  \begin{equation}\label{pertIE}
    \begin{pmatrix}
    z(t)\\
    z'(t)
  \end{pmatrix}
  ={e^{(t-\tau)A}}\begin{pmatrix}
    z(\tau)\\
    z'(\tau)
    \end{pmatrix}
    +{\int_{\tau}^t} e^{(t-s)A}
{
    \begin{pmatrix}
     0\\
      (a-a(s))z'(s)+ (b-b(s))z(s)
    \end{pmatrix}
    }
    \,ds.
  \end{equation}
  for $t_0<\tau<t<t_1$. Set $\chi(t):=(z(t)^2+z'(t)^2)^{1/2}$.
Then \eqref{opnorm} and \eqref{pertIE} imply that
  \[
  \chi(t)\le e^{-\lambda (t-\tau)}{\chi(\tau)}+{2\e}\int_{\tau}^t e^{-\lambda(t-s)} \chi(s)\,ds.
  \]
  The use of the Gronwall inequality yields
  \[
  e^{\lambda(t-\tau)}\chi(t)\le {e^{2\e (t-\tau)}}\chi(\tau),
  \]
which implies \eqref{pertODEdecayest} in the case  $g\equiv 0$.

We prove \eqref{pertODEdecayest} in the case $g\not\equiv 0$.
Fix $\tau\in(t_0,t_1)$, and let $z_0(t)$ be a solution to
\[
z_0''+a(t)z_0'+b(t)z_0=0,\ z_0(\tau)=z(\tau),\ z_0'(\tau)=z(\tau).
\]
For $s\in(\tau,t_1)$, let $W(s,\,\cdot\,)$ be a solution to
\[
\partial_t^2 W(s,t)+a(t)\partial_tW(s,t)+b(t)W(s,t)=0,\ W(s,s)=0,\ \partial_t W(s,s)=g(s).
\]
Then, a solution to $z''+a(t)z'+b(t)z=g$ can be expressed as
$$
z(t)=z_0(t)+\int_{\tau}^t W(s,t)\,ds.
$$
Note that this formula can be checked by direct calculation.

Since
\begin{gather*}
    |z_0(t)|+|z_0'(t)|\le C{e^{-(\lambda-2\e)(t-\tau)}}(|z(\tau)|+|z'(\tau)|)\ \textrm{and}\\
    |W(s,t)|+|\partial_tW(s,t)|\le C{e^{-(\lambda-2\e)(t-s)}}|g(s)|
  \end{gather*}
for $t_0<s<t$ by the case $g\equiv 0$, it follows that
\begin{equation}\label{S3L1E1}
|z(t)|\le C{e^{-(\lambda-2\e)(t-\tau)}}(|z(\tau)|+|z'(\tau)|)+C{\int_{\tau}^t e^{-(\lambda-2\e)(t-s)}}|g(s)|\,ds.
\end{equation}
Also, since
\[
z'(t)=z_0'(t)+W(t,t)+{\int_{\tau}^t}\partial_tW(s,t)\,ds=z_0'(t)+{\int_{\tau}^t}\partial_tW(s,t)\,ds,
\]
it follows that
\begin{equation}\label{S3L1E2}
|z'(t)|\le C{e^{-(\lambda-2\e)(t-\tau)}}(|z(\tau)|+|z'(\tau)|)+C\int_{\tau}^t {e^{-(\lambda-2\e)(t-s)}}|g(s)|\,ds.
\end{equation}
By \eqref{S3L1E1} and \eqref{S3L1E2} we obtain \eqref{pertODEdecayest}.

We finally prove \eqref{pertODEdecayest'}.
If $\lim_{t\to-\infty}z(t)=0$, then there is a sequence $\tau_n\to-\infty$ such that $\lim_{n\to\infty}z'(\tau_n)=0$.
Using \eqref{pertODEdecayest} with $\tau=\tau_n$ and letting $n\to\infty$, we obtain \eqref{pertODEdecayest'}.
\end{proof}

\begin{lemma}\label{linearsolvable}
Suppose that $a(t)$ and $b(t)$ are continuous on $(-\infty,t_1)$ {satisfying} $\lim\limits_{t\to-\infty}a(t)=a$ and $\lim\limits_{t\to-\infty}b(t)=b$.
Then, for a continuous function $h$ on $(-\infty,t_1)$ satisfying $\lim\limits_{t\to-\infty}h(t)=0$, the problem
\begin{equation}\label{linearprob}
  z''+a(t)z'+b(t)z=h\ \textrm{for $t\in(-\infty,t_1)$},\ \lim_{t\to-\infty}z(t)=0
\end{equation}
possesses a unique solution $z\in C^2(-\infty,t_1)$. Furthermore, there are constants {$\tlambda>0$}, $C>0$, and $\tau<t_1$ such that
  \begin{equation}\label{pertODEdecayest''}
    |z(t)|+|z'(t)|\le C\int_{-\infty}^t {e^{-\tlambda(t-s)}}|h(s)|\,ds
  \end{equation}
  for all $t\in(-\infty,\tau)$.
\end{lemma}
\begin{proof}
We prove {only} the existence of solution to problem \eqref{linearprob}, since the estimate \eqref{pertODEdecayest''} and the uniqueness are immediate consequences of \eqref{pertODEdecayest'}.

  Let $\lambda>0$ be as in Lemma~\ref{pertODE} and $\tau<t_1$ be such that $|a(t)-a|<{\lambda/4}$ and $|b(t)-b|<{\lambda/4}$ for all $t<\tau$. For all $\sigma<\tau$, there is a solution $z_\sigma\in C^2(\sigma,\tau]$ to the initial value problem
\[
z_\sigma''+a(t)z_\sigma'+b(t)z_\sigma=h\ \textrm{for $t\in(\sigma,\tau)$},\ z_\sigma(\sigma)=0,\ z_\sigma'(\sigma)=0.
\]
It follows from \eqref{pertODEdecayest} that
\begin{equation}\label{zsigmaunifest}
|z_\sigma(t)|+|z_\sigma'(t)|\le C\int_\sigma^t e^{-\tlambda(t-s)}|h(s)|\,ds\le C\int_{-\infty}^t e^{-\tlambda(t-s)}|h(s)|\,ds\le C\sup_{t<\tau}|h(t)|,
\end{equation}
where {$\tlambda=\lambda/2$}.
By a bootstrap argument, we see that $\{z_\sigma''\}_{\sigma<\tau}$ is uniformly bounded and equi-continuous on any compact subset of $(-\infty,\tau]$.
Let $\{\sigma_n\}\subset\R$ be a sequence such that $\sigma_n\to -\infty$ as $n\to\infty$.
By the Ascoli-Arzel\`{a} theorem on $[\sigma_n,\tau]$ with a diagonal argument we see that there exists a function $z\in C^2(-\infty,\tau]$ such that $z_{\sigma_n}\to z$ in $C^2_{\rm{loc}}(-\infty,\tau]$ as $n\to\infty$.
It is clear that $z''+a(t)z'+b(t)z=h$ for $t\in(-\infty,\tau)$. Furthermore, it follows from \eqref{zsigmaunifest} that $z$ satisfies \eqref{pertODEdecayest''}. In particular,
  \[
  |z(t)|\le C\sup_{s<t}|h(s)|\to 0\ \textrm{as $t\to-\infty$}.
  \]
  Finally, we extend the domain of $z$ to $(-\infty,t_1)$ as a solution $\tilde{z}$ to
  \[
  \tilde{z}''+a(t)\tilde{z}'+b(t)\tilde{z}=h\ \textrm{for $t\in(-\infty,t_1)$},\ \tilde{z}(\tau)=z(\tau),\ \tilde{z}'(\tau)=z(\tau).
  \]
Note that $-\infty<\tau<t_1$ and hence the above problem has a solution.
Then $z$ is a solution to problem \eqref{linearprob} satisfying \eqref{pertODEdecayest''}.
\end{proof}

Now we are ready to prove the unique existence of a singular solution to \eqref{ODE} and their $C^1$ dependence on $\mu$. We temporally allow $\mu$ to be negative.
\begin{theorem}\label{singexists}
  Assume the same conditions as in Theorem~\ref{S1T1}. Then for all $\mu\in\R$, there is a unique solution $u\in C^2(0,\infty)$, denoted by $u^*_\mu$, to the problem
  \begin{equation}\label{ODE'}
    u''+\frac{N-1}{r}u'+K(r)\max\{u,0\}^p+\mu f(r)=0\ \textrm{for $r>0$},\ \lim_{r\to+0}r^{\theta} u(r)=\gamma.
  \end{equation}
  Furthermore, the following assertions hold.
  \begin{enumerate}
    \item $(r,\mu)\mapsto u^*_\mu$ is a $C^1$ function on $(0,\infty)^2$;
    \item Let 
    $$
    \calP_r:=\{\mu\in\R: u^*_\mu(\rho)>0\ \textrm{for all $\rho\in(0,r]$}\}\ \text{ for $r>0$}.
    $$
    Then $\calP_r$ is an open subset of $\R$. Furthermore, $\mu\mapsto u^*_\mu(r)$ and $\mu\mapsto u^{*\prime}_\mu(r)$ are real analytic on each connected component of $\calP_r$.
  \end{enumerate}
\end{theorem}
\begin{remark}
We allow $\mu$ to be negative in order {to show} that $u^*_\mu(r)$ depends analytically near $\mu=0$, which is essential to prove the finiteness of the number of singular solutions to \eqref{ODE'} (see Theorem~\ref{S1T3}).
\end{remark}
\begin{proof}[Proof of Theorem~\ref{singexists}]
 We first prove that a solution $u\in C^2(0,\rho)$ to \eqref{ODE'} can be extended into a solution on $(0,\infty)$.
 Assume on the contrary that the maximal existence time $\rho^*$ for the solution $u$ is finite. Fix $r_1\in(0,\rho^*)$.
 It holds that
 \[
 r^{N-1}u'(r)-r_1^{N-1}u'(r_1)=\int_{r_1}^r (s^{N-1}u'(s))'\,ds\le-\mu\int_{r_1}^r s^{N-1}f(s)\,ds\le Cr^{N-2}
 \]
 for all $r\in(r_1,\rho^*)$. This implies that 
 \[
  u'(r)\le {\frac{C}{r}}+\left(\frac{r_1}{r}\right)^{N-1}u'(r_1)\le {\frac{C}{r}}+u'(r_1),
 \]
 and thus
 \[
 u(r)\le u(r_1)+{C(\log r-\log r_1)}+u'(r_1)(r-r_1),
 \]
 for all $r\in(r_1,\rho^*)$.
 In particular, $\max\{u,0\}$ is bounded on $(r_1,\rho^*)$.
 This implies that
  \[
  |s^{N-1}u'(s)-r^{N-1}u'(r)|\le{\int_{r}^s} |t^{N-1}(-K(t)\max\{u(t),0\}^{p-1}-\mu f(t))|\,dt\le C(s-r)
  \]
  for $r_1<r<s<\rho^*$.
  Hence, $u'$ is uniformly continuous on $[r_1,\rho^*)$, which enables us to extend $u$ into a $C^1$ function on $(0,\rho^*]$.
  Furthermore, by solving an initial value problem starting at $r=\rho^*$, we can extend $u$ into a solution to \eqref{ODE'} on $(0,\rho^*+\varepsilon)$ for some $\varepsilon>0$, which contradicts the definition of $\rho^*$.
  
Let $w(t):=e^{\theta t}u(e^t)$ and $z(t):=w(t)-\gamma$.
Then, \eqref{ODE'} is equivalent to
\begin{equation}\label{zeq'}
\begin{gathered}
  z''+az'+(p\gamma^{p-1}L(t)-A^{p-1})z=R[z]+(k_0-L(t))\gamma^p-\mu e^{(2+\theta)t}f(e^t)\ \textrm{for $t\in\R$},\\
  \lim_{t\to-\infty}z(t)=0,
\end{gathered}
\end{equation}
where
\[
R(z):=-L(t)(\max\{z+\gamma,0\}^p-\gamma^p-p\gamma^{p-1}z).
\]
Let 
$$
{}_0C(-\infty,t]:=\{h\in C(-\infty,t]:\lim\limits_{t\to-\infty}h(s)=0\},
$$
 which is a closed subspace of $L^\infty(-\infty,t)$. Since
  \[
  \lim_{t\to-\infty}p\gamma^{p-1}L(t)-A^{p-1}=(p-1)k_0\gamma^{p-1}>0,
  \]
  it follows from Lemma~\ref{pertODE} that there is $\tau\in\R$ such that for all $h\in {}_0C(-\infty,\tau]$, there is a unique solution $z=Z[h]\in C^2(-\infty,\tau]$ to the problem 
$$
\begin{gathered}
 z''+az'+(p\gamma^{p-1}L(t)-A^{p-1})z=h\ \text{for}\ t\in(-\infty,\tau],\\
\lim_{t\to-\infty}z(t)=0, 
\end{gathered}
$$
and
$$
    \sup_{t\in(-\infty,{\ttau})}\left|Z[h](t)\right|\le C\sup_{t\in(-\infty,{\ttau})}|h(t)|
$$
for all ${\ttau}\in(-\infty,\tau]$.
Furthermore, the problem \eqref{zeq'} on $(-\infty,\tau]$ is equivalent to
  \begin{equation}\label{nonlinearIE}
    z=\Phi[z,\mu]:=Z\left[R(z)+(k_0-L(t))\gamma^p-\mu e^{(2+\theta)t}f(e^t)\right].
  \end{equation}

 For ${\ttau}\in(-\infty,\tau)$ and small $\delta>0$, define
 \[
 {E_{\ttau,\delta}}:=\left\{z\in {}_0C(-\infty,\tau]:\sup_{t\in(-\infty,{\ttau})}|z(t)|\le\delta\right\}.
 \]
 Let $\ttau\in(-\infty,\tau)$ and $z,z_1,z_2\in E_{\ttau,\gamma/2}$. 
Since
  \[
  R(z)=-L(t)\int_0^1 p(\max\{\sigma z+\gamma,0\}^{p-1}-\gamma^{p-1})z\,d\sigma,
  \]
it holds that
\begin{equation*}
|R(z)|\le
C|z|^2.
\end{equation*}
  Furthermore, it follows from
  \begin{align*}
    R(z_1)-R(z_2)&=-L(t)\int_0^1 p((\sigma z_1+\gamma)^{p-1}-(\sigma z_2+\gamma)^{p-1})z_1\,d\sigma\\
    &\qquad -L(t)\int_0^1 p((\sigma z_2+\gamma)^{p-1}-{\gamma^{p-1}})(z_2-z_1)\,d\sigma
  \end{align*}
  that
\begin{equation*}
|R(z_1)-R(z_2)|\le
C(|z_1|+|z_2|)|z_1-z_2|.
\end{equation*}
Thus, there exists a small $\delta>0$ such that for $z,z_1,z_2\in E_{\ttau,\delta}$,
\begin{equation*}
|R(z)|\le \frac{1}{2}\delta\ \ \textrm{and}\ \ |R(z_1)-R(z_2)|\le \frac{1}{2}|z_1-z_2|
\quad\textrm{on {$(-\infty,\ttau]$}.}
\end{equation*}
Hence, for all $\mu\in\R$, there exists $\ttau\in\R$ such that
\[
|k_0-L(t)|\gamma^p+|\mu|e^{(2+\theta)t}f(e^t)<\frac{\delta}{2}\ \textrm{for all $t\in(-\infty,{\ttau}]$},
\]
where \eqref{g0} was used.
By the contraction mapping theorem we obtain a solution $z\in E_{\ttau,\delta}$ to \eqref{nonlinearIE} {which is unique in $E_{\ttau,\delta}$.}
             
  Now we prove the unique existence of solution $z\in C^2(\R)$ to \eqref{zeq'}. The existence is obtained by extending the domain of a solution $z^*_\mu\in {E_{\ttau,\delta}}$ to \eqref{nonlinearIE}. Also, if $z_1,z_2\in C^2(\R)$ are solutions to \eqref{nonlinearIE}, taking $\tau_0\le{\ttau}$ so that $z_1,z_2\in E_{\tau_0,\delta}$, we see that $z_1=z_2$ on $(-\infty,\tau_0]$. By the uniqueness of a solution to an initial value problem for an ODE, we see that $z_1\equiv z_2$. Thus a solution $z\in C^2(\R)$ to \eqref{zeq'} is unique.

  We finally prove (1) and (2). Let $\tau\in\R$. It is easily seen that the map $\Psi:{}_0C(-\infty,\tau]\times\R\to{}_0C(-\infty,\tau]$ defined by $(z,\mu)\mapsto z-\Phi[z,\mu]$ is of the class $C^1$. Furthermore, for all $(z,\mu)\in {}_0C(-\infty,\tau]\times\R$ and $h\in {}_0C(-\infty,\tau]$, the equation $\Psi_z[z,\mu]v=h$, which is equivalent to
  \begin{gather*}
  v''+av'+(p\gamma^{p-1}L(t)-A^{p-1})v=-pL(t)(\max\{z+\gamma\}^{p-1}-\gamma^{p-1})v+h\ \textrm{for $t\in(-\infty,\tau)$},\\
  \lim_{t\to-\infty}v(t)=0,
  \end{gather*}
  is uniquely solvable in ${}_0C(-\infty,\tau]$ by Lemma~\ref{pertODE}. These enable us to apply the implicit function theorem on the map $\Psi$ at $(z^*_\mu,\mu)$ to see that $\mu\mapsto z^*_\mu$ is of the class $C^1$ from $\R$ to ${}_0C(-\infty,\tau]$. Furthermore, since $z_\mu'=\Phi[z_\mu,\mu]'$ and $h\mapsto Z[h]'$ is a bounded linear operator on ${}_0C(-\infty,\tau]$ by \eqref{pertODEdecayest''} in Lemma~\ref{linearsolvable}, we see that $\mu\mapsto z^{*\prime}_\mu$ is also of the class $C^1$ from $\R$ to ${}_0C(-\infty,\tau]$. Thus (1) follows.
  
  (2) follows by the facts that
  \[
{}_0C^{-\gamma}(0,\tau]=\{z\in {}_0C(-\infty,\tau]: z>-\gamma\ \textrm{on $(0,\tau]$}\}
  \]
  is an open subset of ${}_0C(0,\tau]$ and that $\Psi$ is an analytic map on ${}_0C^{-\gamma}(0,\tau]\times\R$.
  \end{proof}
\bigskip

In the following, we denote by $u(\cdot,\zeta)$ a regular solution to \eqref{ODE} such that $u(0,\zeta)=\zeta$ for $\zeta>0$, and define $r_0(\zeta)\in(0,\infty]$ as the maximal existence radius of $u(\cdot,\zeta)$, {\it i.e.}, $u(r_0(\zeta),\zeta)=0$ if $r_0(\zeta)<\infty$. Also, let $\overline{u}(\cdot,\zeta)$ be a solution to the problem
  \begin{equation}\label{ODE0}
\begin{cases}
    \overline{u}''+\frac{N-1}{r}\overline{u}'+k_0r^\alpha\overline{u}^p=0 & \textrm{for}\ r>0,\\
    \overline{u}(0)=\zeta,\\
    \overline{u}(r)>0 &\textrm{for}\ r>0.
    \end{cases}
  \end{equation}
  Note that
  \begin{equation}\label{ODE0rem}
    \begin{gathered}
      \overline{u}(r,\zeta)=\zeta\overline{u}(\zeta^{1/\theta}r,1)\ \textrm{for $\zeta>0$, $r>0$},\\
      \lim_{r\to\infty}r^{\theta}\overline{u}(r,\zeta)=\gamma,\ \lim_{r\to\infty}r^{1+\theta}\overline{u}_r(r,\zeta)=-\theta\gamma,\ \textrm{for $\zeta>0$.}
    \end{gathered}
  \end{equation}
 
\begin{lemma}\label{uapprox1}
  Assume the same conditions as in Theorem~\ref{S1T1}.
  Then for all $\sigma>0$, there is $\zeta_\sigma>0$ such that $r_0(\zeta)\ge \sigma\zeta^{-1/\theta}$ for all $\zeta>\zeta_\sigma$. Furthermore,
  \begin{gather}
    \lim_{\zeta\to\infty}\zeta^{-1}|u(\sigma\zeta^{-1/\theta},\zeta)-\overline{u}(\sigma\zeta^{-1/\theta},\zeta)|=0,\label{approx}\\
    \lim_{\zeta\to\infty}\zeta^{-1-1/\theta}|u_r(\sigma\zeta^{-1/\theta},\zeta)-\overline{u}_r(\sigma\zeta^{-1/\theta},\zeta)|=0,\label{approx'}
  \end{gather}
  for all $\sigma>0$.
\end{lemma}

\begin{proof}
  Let $\zeta>0$ and $r\in(0,r_0(\zeta)]$.
Subtracting \eqref{ODE0} from \eqref{ODE} and integrating it, we see that
  \begin{align*}
    u(r,\zeta)-\overline{u}(r,\zeta)
    &=-\int_0^r s^{1-N}\int_0^s t^{N-1}(K(t)u(t,\zeta)^p-k_0t^\alpha\overline{u}(t,\zeta)^p-\mu f(t))\,dt\,ds\\
    &=-\frac{1}{N-2}\int_0^r (t^{2-N}-r^{2-N})t^{N-1}(K(t)u(t,\zeta)^p-k_0t^\alpha\overline{u}(t,\zeta)^p-\mu f(t))\,dt.
  \end{align*}
  This together with $u(r,\zeta),\overline{u}(r,\zeta)\le\zeta$ and the mean value theorem implies that
  \begin{align*}
    |u(r,\zeta)-\overline{u}(r,\zeta)|
    &\le\frac{1}{N-2}\int_0^r  k_0t^{1+\alpha}|u(t,\zeta)^p-\overline{u}(t,\zeta)^p|\,dt\\
    &\qquad+\frac{1}{N-2}\int_0^r t^{1+\alpha}|t^{-\alpha} K(t)-k_0|u(t,\zeta)^p\,dt+\mu\int_0^r sf(s)\,ds\\
    &\le\frac{1}{N-2}\int_0^r pk_0t^{1+\alpha}\zeta^{p-1}|u(t,\zeta)-\overline{u}(t,\zeta)|\,dt+\Phi(r,\zeta),
  \end{align*}
  where
  \[
  \Phi(r,\zeta):=\frac{1}{N-2}\int_0^r t^{1+\alpha}|t^{-\alpha} K(t)-k_0|\zeta^p\,dt+\mu\int_0^{r}sf(s)\,ds.
  \]
  By the Gronwall inequality, we obtain
  \begin{equation}\label{approx1}
    |u(r,\zeta)-\overline{u}(r,\zeta)|\le \Phi(r,\zeta)e^{C_1\zeta^{p-1}r^{2+\alpha}},
  \end{equation}
  where $C_1$ is independent of $r,\zeta$. Furthermore, it follows from \eqref{K0} and \eqref{f0} that
  \begin{equation}\label{approx2}
      \Phi(r,\zeta)= o(1)\zeta^pr^{2+\alpha}+o(1)\ \textrm{as $r\to 0$ uniformly in $\zeta$}.
  \end{equation}
  Thus, for all $\sigma>0$ and $\varepsilon>0$, choosing sufficiently large $\zeta^\varepsilon>0$, we deduce from \eqref{approx1} and \eqref{approx2} that
\begin{equation}\label{S3L5E1}
  |u(r,\zeta)-\overline{u}(r,\zeta)|\le e^{C_1\sigma^{2+\alpha}}\varepsilon(\sigma^{2+\alpha}\zeta+1)\le 2\sigma^{2+\alpha}e^{C_1\sigma^{2+\alpha}}\varepsilon\zeta
\end{equation}
  for all $\zeta>\zeta^\varepsilon$ and $r\in(0,\min\{r_0(\zeta),\sigma\zeta^{-1/\theta}\}]$. Since $\overline{u}(r,\zeta)=\zeta\overline{u}(1,\zeta^{1/\theta})>\zeta\overline{u}(1,\sigma)$, selecting $\varepsilon>0$ so that $2\sigma^{2+\alpha}e^{C_1\sigma^{2+\alpha}}\varepsilon<\overline{u}(1,\sigma)$, we obtain
  \[
  u(\min\{r_0(\zeta),\sigma\zeta^{-1/\theta}\},\zeta)\ge \zeta\overline{u}(1,\sigma)-2\sigma^{2+\alpha}e^{C_1\sigma^{2+\alpha}}\varepsilon\zeta>0,
  \]
  and hence $r_0(\zeta)>\sigma\zeta^{-1/\theta}$ for all $\zeta>\zeta^\varepsilon$.
We define $\zeta_{\sigma}$ by
$$
\zeta_{\sigma}:=\zeta^{\e}.
$$
  Moreover, \eqref{S3L5E1} implies \eqref{approx}.

  Also, subtracting \eqref{ODE0} from \eqref{ODE} and integrating it, we have
  \[
  u_r(r,\zeta)-\overline{u}_r(r,\zeta)=-{r^{1-N}}\int_0^rs^{N-1}(K(s)u(s,\zeta)^p-k_0s^{\alpha}\overline{u}(s,\zeta)^p)\,ds-\mu r^{1-N}\int_0^{r}s^{N-1}f(s)\,ds.
  \]
  This together with $u(r,\zeta),\overline{u}(r,\zeta)\le\zeta$, \eqref{approx1} and the mean value theorem implies that for $r\le\sigma\zeta^{-1/\theta}$,
  \begin{align*}
    |u_r(r,\zeta)-\overline{u}_r(r,\zeta)|
    &\le r^{1-N}\int_0^r k_0s^{N-1+\alpha}|u(s,\zeta)^p-\overline{u}(s,\zeta)^p|\,ds\\
    &\qquad+r^{1-N}\int_0^r s^{N-1+\alpha}|s^{-\alpha} K(s)-k_0|u(s,\zeta)^p\,ds+\mu r^{1-N}\int_0^{r}s^{N-1}f(s)\,ds\\
    &\le r^{1-N}\int_0^r pk_0s^{N-1+\alpha}\cdot C\zeta^{p-1}\Phi(r,\zeta)\,ds+\Psi(r,\zeta)\\
    &\le C\zeta^{p-1}r^{1+\alpha}\Phi(r,\zeta)+\Psi(r,\zeta),
  \end{align*}
where $C$ is independent of $\zeta$ and
  \[
  \Psi(r,\zeta):=r^{1-N}\int_0^r s^{N-1+\alpha}|s^{-\alpha} K(s)-k_0|\zeta^p\,ds+\mu{r^{-1}}\int_0^{r}sf(s)\,ds.
  \]
  It follows from \eqref{K0} and \eqref{f0} that
  \begin{equation*}
    \Psi(r,\zeta)=\zeta^pr^{1+\alpha}o(1)+o(r^{-1})\ \textrm{as $r\to 0$ uniformly in $\zeta$.}
  \end{equation*}
  Hence, noting that $p-1-(1+\alpha)/\theta=1/\theta$, we obtain
  \begin{align*}
  |u_r(\sigma\zeta^{-1/\theta},\zeta)-\overline{u}_r(\sigma\zeta^{-1/\theta},\zeta)|
  &{\le C\zeta^{p-1}(\sigma\zeta^{-1/\theta})^{1+\alpha}\Phi(\sigma\zeta^{-1/\theta},\zeta)+\Psi(\sigma\zeta^{-1/\theta},\zeta)}\\
  &=o(\zeta^{1+1/\theta})+o(\zeta^{1+1/\theta})+o({\zeta^{1/\theta}})\\
  &=o(\zeta^{1+1/\theta})
  \end{align*}
  as $\zeta\to\infty$, and thus \eqref{approx'} holds.
\end{proof}

\begin{lemma}\label{uapprox2}
  Assume the same conditions as in Theorem~\ref{S1T1}. There are $\varepsilon_0>0$ and $C_0>0$ such that the following holds:
  For all $\varepsilon\in(0,\varepsilon_0]$, there is $\rho>0$ such that if $r_1\in(0,\rho)$ satisfies
  \begin{equation}\label{trap1}
    |u(r_1,\zeta)-\gamma r_1^{-\theta}|\le\varepsilon r_1^{-\theta}\ \textrm{and}\ |u_r(r_1,\zeta)+\theta\gamma r_1^{-1-\theta}|\le\varepsilon r_1^{-1-\theta},
  \end{equation}
  then $r_0(\zeta)>\rho$ and
  \begin{equation*}
    |u(r,\zeta)-\gamma r^{-\theta}|\le C_0\varepsilon r^{-\theta},\ |u_r(r,\zeta)+\theta\gamma r^{-1-\theta}|\le C_0\varepsilon r^{-1-\theta},
  \end{equation*}
  for all $r\in (r_1,\rho]$.
\end{lemma}
\begin{proof}
  Let $w_\zeta(t):=e^{\theta t}u(e^{t},\zeta)$ and $\overline{w}_\zeta(t):=w_\zeta(t)-\gamma$ for $\zeta>0$ and $t\in\R$ for which $u(e^t,\zeta)>0$.
Because of the assumption of the lemma especially \eqref{trap1}, for small $\e>0$,
\begin{equation}\label{trap1'}
    |\overline{w}_\zeta(t_0)|\le\varepsilon,\ |\overline{w}'_\zeta(t_0)|\le(\theta+1)\varepsilon,
  \end{equation}
where $t_1:=\log r_1$.
Furthermore, $\overline{w}_\zeta$ is a solution to
  \begin{equation*}
    \overline{w}_\zeta''+a\overline{w}_\zeta'+b(t)\overline{w}_\zeta=-(L(t)-k_0)\gamma^p-\mu g(t),
  \end{equation*}
  where
  \[
  b(t):=-A^{p-1}+L(t)\int_0^1 p(\gamma+\sigma\overline{w}_{\zeta}(t))^{p-1}\,d\sigma.
  \]
Because of the continuity of the solution $\overline{w}_{\zeta}$, for each small $\delta\in (0,\gamma)$, 
we see that the set $\{ \overline{t}\in (t_1,\infty]:\ |\overline{w}_{\zeta}(t)|\le \delta\ \ \text{for}\ \ t_1\le t\le \overline{t} \}$ is not empty.
Let
\begin{equation}\label{S3L6E1}
t_2:=\sup\{ \overline{t}\in (t_1,\infty]:\ |\overline{w}_{\zeta}(t)|\le \delta\ \ \text{for}\ \ t_1\le t\le \overline{t} \}.
\end{equation}

It follows from \eqref{S3L6E1} that
  \begin{align*}
  |b(t)-(p-1)A^{p-1}|&=p\left|L(t)\int_0^1(\gamma+\sigma\overline{w}_{\zeta}(t))^{p-1}\,d\sigma-k_0\gamma^{p-1}\right|\\
  &\le p|L(t)-k_0|\int_0^1(\gamma+\sigma\overline{w}_{\zeta}(t))^{p-1}\,d\sigma+pk_0\int_0^1|(\gamma+\sigma\overline{w}_{\zeta}(t))^{p-1}-\gamma^{p-1}|\,d\sigma\\
  &\le C|L(t)-k_0|+C\delta
\end{align*}
for $t_1\le t\le t_2$.
This together with Lemma~\ref{pertODE} and \eqref{trap1'} implies that there exist $\tau_0\in\R$ (independent of $\varepsilon_0$) and $\lambda>0$ such that
  \begin{equation}\label{wbarest1}
  |\overline{w}_{\zeta}(t)|+|\overline{w}_{\zeta}'(t)|\le Ce^{-\lambda(t-t_1)}\varepsilon+C\int_{t_1}^te^{-\lambda(t-s)}(\gamma^{p-1}|L(s)-k_0|+|g(s)|)\,ds
  \end{equation}
  for all $t\in(t_1,\min\{t_2,\tau_0,\log r_0(\zeta)\}]$, provided that $\varepsilon_0$ is sufficiently small and that $t_1<\tau_0$.

  It follows from \eqref{g0} that
  \[
  \int_{-\infty}^t e^{-\lambda(t-s)}|g(s)|\,ds\le\int_{-\infty}^t e^{\theta(t-s)}|g(s)|\,ds\le Ce^{\theta t}
  \]
  for $t>t_1$. This together with \eqref{L0} implies the existence of $\tau_1\in(t_1,\tau_0)$ such that
  \[
  \int_{t_1}^te^{-\lambda(t-s)}(\gamma^{p-1}|L(s)-k_0|+|g(s)|)\,ds\le\varepsilon
  \]
  for all $t_1<t<\tau_1$. Combining this and \eqref{wbarest1}, we obtain
\begin{equation}\label{S3L6E2}
  |\overline{w}_{\zeta}(t)|+|\overline{w}_{\zeta}'(t)|\le C_0\varepsilon
\end{equation}
  for all $t\in(t_1,\min\{t_2,\tau_1,\log r_0(\zeta)\}]$.
We define $\e_0$ by
$$
\e_0:=\frac{\delta}{2C_0}.
$$
Then, if $0<\e<\e_0$, then \eqref{S3L6E2} holds for $t_1\le t\le\min\{\tau_1,\log r_0(\zeta)\}$, otherwise $t_2<\min\{\tau_1,\log r_0(\zeta)\}$ and $|\overline{w}_{\zeta}(t_2)|\ge \delta$ (by \eqref{S3L6E1}) which contradicts \eqref{S3L6E2}.
  
  This implies that
  \[
|u(r,\zeta)-\gamma r^{-\theta}|\le C_0\varepsilon r^{-\theta},\ \ 
|u_r(r,\zeta)+\theta\gamma r^{-1-\theta}|\le C_0\varepsilon r^{-1-\theta},
  \]
  for all $r\in(r_1,\min\{\rho,r_0(\zeta)\}]$, where we define $\rho$ by 
$$
\rho:=e^{\tau_1}.
$$  
  Moreover, since
  \[
  u(\min\{\rho,r_0(\zeta)\},\zeta)\ge \gamma r^{-\theta}-C_0\e r^{-\theta}>\left(\gamma-\frac{\delta}{2}\right)r^{-\theta}>0,
  \]
  we see that $r_0(\zeta)>\rho$.
\end{proof}

\begin{theorem}\label{convtosing}
Assume the same conditions as in Theorem~\ref{S1T1}. There exists $r^*>0$ such that \eqref{ODE} has a unique singular solution $u^*(r)$ for $0<r<r^*$ and that
the regular solution $u(r,\zeta)$ satisfies
\[
u(\,\cdot\,,\zeta)\to u^*\ \ \textrm{in}\ \ C^2_{\rm loc}(0,r^*)\ \ \textrm{as}\ \zeta\to\infty.
\]
\end{theorem}
\begin{proof}
Let $\e_0>0$ and $C_0>0$ be as in Lemma~\ref{uapprox2}.
Let $r^*:=\rho>0$ be as in Lemma~\ref{uapprox2} with $\varepsilon=\min\{\gamma/(2C_0),\gamma/(2\theta C_0),\varepsilon_0\}$. By \eqref{ODE0rem}, we find $\sigma_0>1$ such that
  \begin{equation*}
    |\overline{u}(\sigma_0\zeta^{-1/\theta},\zeta)-\gamma\sigma_0^{-\theta}\zeta|\le\frac{\varepsilon}{2}\sigma_0^{-\theta}\zeta,\ |\overline{u}_r(\sigma_0\zeta^{-1/\theta},\zeta)+\gamma\theta\sigma_0^{-1-\theta}\zeta^{1+1/\theta}|\le\frac{\varepsilon}{2}\sigma_0^{-1-\theta}\zeta^{1+1/\theta},
  \end{equation*}
  for all $\zeta>0$. Combining this with Lemma~\ref{uapprox1}, we obtain
$$
    |u(\sigma_0\zeta^{-1/\theta},\zeta)-\gamma\sigma_0^{-\theta}\zeta|\le\varepsilon\sigma_0^{-\theta}\zeta,\ |u_r(\sigma_0\zeta^{-1/\theta},\zeta)+\gamma\theta\sigma_0^{-1-\theta}\zeta^{1+1/\theta}|\le\varepsilon\sigma_0^{-1-\theta}\zeta^{1+1/\theta},
$$
  for sufficiently large $\zeta$.
  It follows from Lemma~\ref{uapprox2} that $r_0(\zeta)>r^*$ and
  \begin{equation}\label{u-gamma'}
    |u(r,\zeta)-\gamma r^{-\theta}|\le{C_0\frac{\gamma}{2C_0}}r^{-\theta},\ 
    |u_r(r,\zeta)+\gamma\theta r^{-1-\theta}|\le C_0{\frac{\gamma\theta}{2C_0}}r^{-1-\theta},
  \end{equation}
  for $r\in(\sigma_0\zeta^{-1/\theta},r^*]$.

  The inequalities \eqref{u-gamma'} yield the boundedness and equicontinuity of $u(\cdot,\zeta)$ for sufficiently large $\zeta>0$ on each compact subset of $(0,r^*]$. By the Ascoli-Arzel\'a theorem with a diagonal argument, there is a sequence $\{\zeta_n\}_{n=1}^\infty\subset(0,\infty)$ such that $\zeta_n\to\infty$ as $n\to\infty$ and that $u(\cdot,\zeta)$ converges to some $u^*$ in $C^2_{\textrm{loc}}(0,r^*]$ as $n\to\infty$. Since
  \begin{equation}\label{eq:u*bddbelow}
  u^*(r)\ge \frac{\gamma}{2}r^{-\theta}\ \textrm{for all $r\in(0,\rho]$}
  \end{equation}
  by \eqref{u-gamma'}, $u^*$ is a singular solution to \eqref{ODE}. This together with the uniqueness of a singular solution to \eqref{ODE} (see Theorem~\ref{singexists}) implies that
  \[
  u(\,\cdot\,,\zeta)\to u^*\ \ \textrm{in}\ \ C^2_{\rm loc}(0,\rho]\ \ \textrm{as}\ \zeta\to\infty.
  \]
\end{proof}

\section{Intersection number}
Let $\overline{u}(r,\zeta)$ denote the solution of \eqref{ODE0}.
The function $\gamma r^{-\theta}$ is the singular solution of the equation in \eqref{ODE0}.
Then, we recall the following:
If $p_S(\alpha)<p<p_{JL}(\alpha)$, then there are $0=\sigma_0<\sigma_1<\sigma_2<\ldots$ such that $\lim\limits_{n \to\infty}\sigma_n=\infty$ and
\begin{equation}\label{ubarintersect}
\begin{gathered}
  \overline{u}(\sigma_n,1)=\gamma \sigma_n^{-\theta}\ \textrm{for $n\in\Z_{\ge 1}$},\\
  (-1)^n\overline{u}(\sigma,1)<(-1)^n\gamma \sigma^{-\theta}\ \textrm{for $\sigma\in(\sigma_{n},\sigma_{n+1})$, $n\in\Z_{\ge 0}$}.
\end{gathered}
\end{equation}
Roughly speaking, the intersection number between $\overline{u}(\sigma,1)$ and $\gamma\sigma^{-\theta}$ in $0<\sigma<\infty$ is infinite.
See \cite[Proposition 3.5]{W93} and \cite[Theorem A]{MT17} for details.

\begin{theorem}\label{inftyintersect}
Assume the same conditions as in Theorem~\ref{S1T2}. Let $u^*$ be a singular solution of \eqref{ODE} on $(0,r^*)$. Then for each $\rho\in (0,r^*]$,
\[
\calZ_{(0,\rho)}\left[u(\,\cdot\,,\zeta)-u^*\right]\to\infty\ \ \textrm{as}\ \ \zeta\to\infty.
\]
\end{theorem}
\begin{proof}
  We fix $\sigma'_n\in(\sigma_{n-1},\sigma_n)$ for $n\in\Z_{\ge 1}$. It follows from Lemma~\ref{u0asypro} and \eqref{approx} that
  \[
    u(\sigma\zeta^{-1/\theta},\zeta)-u^*(\sigma\zeta^{-1/\theta})=\overline{u}(\sigma\zeta^{-1/\theta},\zeta)-\gamma\sigma^{-\theta}\zeta+o(\zeta)=\zeta(\overline{u}(\sigma,1)-\gamma\sigma^{-\theta}+o(1))
  \]
  as $\zeta\to+\infty$ for each $\sigma>0$. This together with \eqref{ubarintersect} implies that there are $0<\overline{\zeta}_1<\overline{\zeta}_2<\ldots$ such that
  \[
  (-1)^nu(\sigma'_n\zeta^{-1/\theta},\zeta)<(-1)^nu^*(\sigma'_n\zeta^{-1/\theta})
  \]
  for all $n\in\Z_{\ge 1}$ and $\zeta>\overline{\zeta}_n$. By the intermediate value theorem, for all $n\in\Z_{\ge 1}$, $\zeta>\overline{\zeta}_n$, and $k\in\{1,\ldots,n\}$, there is $r^\zeta_k\in(\sigma'_{k-1}\zeta^{-1/\theta},\sigma'_k\zeta^{-1/\theta})$ such that $u(r^\zeta_k,\zeta)=u^*(r^\zeta_k)$. This implies that
  \[
  \calZ_{(0,\rho)}[u(\cdot,\zeta)-u^*]\ge \max\{k:\sigma'_k\zeta^{-1/\theta}\le\rho\}
  \]
  for all $\zeta>\overline{\zeta}_n$. In particular,
  \[
  \calZ_{(0,\rho)}[u(\cdot,\zeta)-u^*]\ge n\ \textrm{if $\zeta>\max\{(\rho^{-1}\sigma_n)^{\theta},\overline{\zeta}_n\}$}
  \]
  for all $n\in\Z_{\ge 1}$. The proof of Theorem~\ref{inftyintersect} is complete.
\end{proof}

\section{Proof of Theorems~\ref{S1T1} and \ref{S1T2}}
\begin{proof}[Proof of Theorem~\ref{S1T1}]
Let $\{u_j\}_{j=0}^{\infty}$ be bounded solutions of \eqref{ODE} such that $\left\|u_j\right\|_{\infty}\to\infty$ as $j\to\infty$.
It follows from Theorem~\ref{convtosing} that there exists a solution $u^*(r)\in C^2(0,r^*]$ of \eqref{ODE} with some $r^*>0$ such that $u^*(r)\to\infty$ as $r\to 0$ and
$$
u_j\to u^*\ \ \textrm{in}\ \ C^2_{loc}(0,r^*]\ \ \textrm{as}\ \ j\to\infty.
$$
Since every $u_j(r)$ is defined for $r>0$, we can extend a domain of the singular solution $u^*(r)$ as follows.
Because of the continuous dependence of the solutions to ODEs on initial values, we see that $u_j$ converges to some $U$ in $C^2_{\textrm{loc}}[r^*,\infty)$, where $U$ is a solution to the problem
\[
\begin{cases}
  U''+\frac{N-1}{r}U'+K(r)\max\{U,0\}^p+\mu f(r)=0\ \textrm{for $r\in(r^*,\infty)$},\\
  U(r^*)=u^*(r^*),\ U'(r^*)=u^{*\prime}(r^*).
\end{cases}
\]
We define $u^*\equiv U$ on $(r^*,\infty)$. 

We prove the assertion of Theorem~\ref{S1T1} by contradiction.
Suppose the contrary, {\it i.e.,} $u^*(r)$ has a finite first positive zero $r_0\in (0,\infty)$.
By Hopf's lemma we see that $u^{*\prime}(r_0)<0$.
Therefore, there exists a small $\e>0$ such that $u^*(r_0+\e)<0$.
This is a contradiction, because $u_j(r_0+\e)\to u^*(r_0+\e)$ as $j\to\infty$ and $u_j(r_0+\e)>0$.
The proof is complete.
\end{proof}

\begin{proof}[Proof of Theorem~\ref{S1T2}]
We prove the theorem by contradiction.
Suppose the contrary, {\it i.e.,}
\begin{equation}\label{T2PE0}
\textrm{there exists $\zeta_0>0$ such that}\ 
\left\|u\right\|_{\infty}<\zeta_0\ \ \textrm{for every bounded solution $u$ of \eqref{ODE}.}
\end{equation}
We see that $u(r,\zeta)$ has a finite first positive zero for every $\zeta\ge\zeta_0+1$, {\it i.e.,}
\[
r_0(\zeta)<\infty.
\]
We extend a domain of the solution $u(r,\zeta)$ by defining it as the solution to the initial value problem \eqref{S2E0}.
The solution $u(r,\zeta)$ is a $C^1$ function in $(r,\zeta)$.
By Hopf's lemma we see that $u_r(r_0(\zeta),\zeta)<0$.
It follows from the implicit function theorem that the first positive zero $r_0(\zeta)$ is a $C^1$-function of $\zeta$.
Since $r_0(\zeta)<\infty$ for $\zeta\ge\zeta_0+1$, $r_0(\zeta)$ is a $C^1$-function defined on $(\zeta_0+1,\infty)$.

We consider the intersection number
\[
\calZ_{(0,r_0(\zeta))}\left[u^*(\,\cdot\,)-u(\,\cdot\,,\zeta)\right].
\]
Each zero of $u^*(r)-u(r,\zeta)$ is nondegenerate, because of the uniqueness of a solution for an ODE of the second order.
We show by contradiction that the zero set does not have an accumulation point.
Suppose that $r_*\in (0,\infty)$ is an accumulation point of the zero set of $u^*(r)-u(r,\zeta)$.
Then there exists $\{r_k\}_{k=1}^{\infty}\subset (0,\infty)$ such that $u^*(r_k)-u(r_k,\zeta)=0$ and that one of the following holds:
\begin{align*}
\textrm{(i)}\ & r_1<r_2<\cdots <r_k<\cdots\ \ \textrm{and}\ \ r_k\to r_*,\\
\textrm{(ii)}\ & r_1>r_2>\cdots >r_k>\cdots\ \ \textrm{and}\ \ r_k\to r_*.
\end{align*}
Then, by Rolle's theorem, there exists $\{s_k\}_{k=0}^{\infty}\subset (0,\infty)$ such that $s_k$ is in between $r_k$ and $r_{k+1}$, $u^{*\prime}(s_k)-u_r(s_k,\zeta)=0$ for all $k\ge 0$ and $s_k\to r_*$ as $k\to\infty$.
Since $u^*(r)$ and $u(r,\zeta)$ are $C^1$-functions of $r$, we have
\[
u^*(r_*)-u(r_*,\zeta)=0\ \ \textrm{and}\ \ u^{*\prime}(r_*)-u_r(r_*,\zeta)=0.
\]
It follows from the uniqueness of a solution for an ODE of the second order that $u^*(r)\equiv u(r,\zeta)$, which is a contradiction.
Thus, the zero set does not have an accumulation point.
Since $r_0(\zeta)<\infty$ for each fixed $\zeta\ge\zeta_0+1$, we see that
\[
\calZ_{(0,r_0(\zeta))}\left[u^*(\,\cdot\,)-u(\,\cdot\,,\zeta)\right]<\infty.
\]

Because of Theorem~\ref{convtosing}, there exists $\rho>0$ such that
\[
u(r,\zeta)\to u^*(r)\ \ \textrm{in}\ \ C^2_{loc}(0,\rho]\ \ \textrm{as}\ \ \zeta\to\infty.
\]
This indicates that there exists $\zeta_1\in (\zeta_0+1,\infty)$ such that $r_0(\zeta)>\rho$ for all $\zeta\ge\zeta_1$.
Let
\[
m_1:=\calZ_{(0,r_0(\zeta_1))}\left[u^*(\,\cdot\,)-u(\,\cdot\,,\zeta_1)\right].
\]
By Theorem~\ref{inftyintersect} we see that
\[
\calZ_{(0,\rho)}\left[u^*(\,\cdot\,)-u(\,\cdot\,,\zeta)\right]\to\infty
\ \ \textrm{as}\ \ \zeta\to\infty.
\]
Therefore, there exists $\zeta_2\in (\zeta_1,\infty)$ such that
\[
m_2:=\calZ_{(0,\rho)}\left[u^*(\,\cdot\,)-u(\,\cdot\,,\zeta_2)\right]>m_1.
\]
On the other hand, we have
\begin{equation}\label{T2PE1}
\lim_{r\to 0}\left(u^*(r)-u(r,\zeta)\right)>0\ \ \textrm{and}\ u^*(r_0(\zeta))-u(r_0(\zeta),\zeta)=u^*(r_0(\zeta))>0
\ \textrm{for}\ \zeta\ge\zeta_0+1.
\end{equation}
Since each zero of $u^*(r)-u(r,\zeta)$ is nondegenerate, each zero depends continuously on $\zeta$.
Moreover, a zero does not split or several zeros do not merge into one zero.
Another zero does not come from the boundary of the interval $(0,r_0(\zeta))$, because of (\ref{T2PE1}).
Since a zero number is finite for each fixed $\zeta$, $\calZ_{(0,r_0(\zeta))}\left[u^*(\,\cdot\,)-u(\,\cdot\,,\zeta)\right]$ does not change in a small neighborhood of $\zeta$.

Let
\[
\zeta_3:=\inf\{\zeta>\zeta_1: \calZ_{(0,r_0(\zeta))}\left[u^*(\,\cdot\,)-u(\,\cdot\,,\zeta)\right]=m_2\}.
\]
Because of the definition of $\zeta_3$, $\calZ_{(0,r_0(\zeta))}\left[u^*(\,\cdot\,)-u(\,\cdot\,,\zeta)\right]$ is not constant for $\zeta$ in any small neighborhood of $\zeta_3$.
In that case $r_0(\zeta_3)=\infty$, and hence $u(r,\zeta_3)$ is a bounded solution, which contradicts \eqref{T2PE0}.
Note that if $r_0(\zeta_3)<\infty$, then $\calZ_{(0,r_0(\zeta))}\left[u^*(\,\cdot\,)-u(\,\cdot\,,\zeta)\right]$ is constant in a neighborhood of $\zeta_3$, which contradicts the definition of $\zeta_3$.
The assertion (\ref{T2PE0}) does not hold.
Since a set of a bounded solutions of \eqref{ODE} is not uniformly bounded, the conclusion of the theorem holds.
The proof is complete.
\end{proof}

\section{Proof of Theorem~\ref{S1T3}}
Let $u^*_\mu$ and $\calP_r$ be as in Theorem~\ref{singexists}. We also define
\begin{gather*}
\calP_\infty:=\bigcap_{r>0}\calP_r=\{\mu\in\R: u^*_\mu(r)>0\ \textrm{for all $r>0$}\},\\
\calF:=\{\mu\in\calP_\infty\cap[0,\infty): u^*_\mu\ \textrm{is fast-decay}\},\\
\calS:=\{\mu\in\calP_\infty\cap[0,\infty):u^*_\mu\ \textrm{is slow-decey}\},\\
\calB:=[0,\infty)\setminus\calP_{\infty}.
\end{gather*}
As we will see in the proof of Theorem~\ref{S1T3}, a slow-decay solution satisfies $u(r)=O(r^{-\frac{2+\beta}{p-1}})$.
A fast-decay solution satisfies $u(r)=O(r^{2-N})$.
Fast-decay solutions always decay faster than slow-decay solutions, otherwise $N-2\ge (2+\beta)/(p-1)$ and $(N+2+2\beta)/(N-2)<p\le (N+\beta)/(N-2)$, which contradicts $\beta>-2$.

Note that, if we fix a constant $\mu_0>0$, for $\mu\in[-\mu_0,\mu_0]$, all constants within the estimates in Sections 2--5 can be taken independently of $\mu$.
\begin{lemma}\label{mubound}
Assume the same conditions as in Theorem~\ref{S1T3}. There exists $\mu^*>0$ such that $\mu\notin\calP_\infty$ if $\mu\ge\mu^*$. 
\end{lemma}
\begin{proof}
This is an immediate consequence of \cite[Theorem 1.1]{IK24}.
\end{proof}

\begin{lemma}\label{SFsollimit}
  Assume the same conditions as in Theorem~\ref{S1T3}. For all $\mu\in\calF$, the limit
  \[
  \eta(\mu):=\lim_{r\to\infty}r^{N-2}u^*_\mu(r)
  \]
  exists and satisfies
  \begin{align}
    u^{*\prime}_\mu(r)&=r^{1-N}\left\{-(N-2)\eta(\mu)+\int_r^\infty s^{N-1}(K(s)u^{*}_\mu(s)^p+\mu f(s))\,ds\right\}, \label{u'inftyformula}\\
    u^*_\mu(r)&=r^{2-N}\eta(\mu)-\frac{1}{N-2}\int_r^\infty (r^{2-N}-t^{2-N})t^{N-1}(K(t)u^{*}_\mu(t)^p+\mu f(t))\,dt.\label{uinftyformula}
  \end{align}
  Furthermore, the function $r^{N-2}u^*_\mu(r)$ is strictly increasing in $r$.
\end{lemma}
  \begin{proof}
    We set $H(r):=K(r)u^{*}_\mu(r)^p+\mu f(r)$. By $u^*_\mu(r)=O(r^{2-N})$ as $r\to\infty$, \eqref{finfty}, and the fact that $p>(N+\beta)/(N-2)$, we see that, as $r\to\infty$,
    \begin{align}
    \int_r^\infty s^{N-1}H(s)\,ds &\le C\int_r^\infty (s^{N-1+\beta+p(2-N)}+s^{N-1}f(s))\,ds=o(1),\label{S6L2E1}\\
    \int_r^\infty sH(s)\,ds &=o(r^{2-N}).\label{S6L2E2}
    \end{align}
It follows from \eqref{ODE} that
    \begin{gather}
      \rho^{N-1}u^{*\prime}_\mu(\rho)=r^{N-1}u^{*\prime}_\mu(r)-\int_r^\rho s^{N-1}H(s)\,ds,\nonumber\\
      u^*_\mu(\rho)=u^*_\mu(r)+\frac{1}{N-2}(r-\rho^{2-N}r^{N-1})u^{*\prime}_\mu(r)-\frac{1}{N-2}\int_r^\rho (\rho^{2-N}-s^{2-N})s^{N-1}H(s)\,ds
        \label{uinftyformula2}
    \end{gather}
for all $r<\rho$.
Because of \eqref{S6L2E2}, \eqref{uinftyformula2} is equivalent to
    \begin{multline}\label{S6L2E3}
      (N-2)\rho^{N-2}u^*_\mu(\rho)+r^{N-1}u^{*\prime}_\mu(r)+\int_r^\rho s^{N-1}H(s)\,ds+\rho^{N-2}\int_\rho^\infty sH(s)\,ds\\
=\rho^{N-2}\left\{(N-2)u^*_\mu(r)+ru^{*\prime}_\mu(r)+\int_r^\infty sH(s)\,ds\right\}.
\end{multline}
By \eqref{S6L2E1} and \eqref{S6L2E2} we see that the LHS of \eqref{S6L2E3} is bounded as $\rho\to\infty$.
Hence by \eqref{S6L2E3} we deduce that
     \[
     (N-2)u^*_\mu(r)+ru^{*\prime}_\mu(r)+\int_r^\infty sH(s)\,ds=0,
     \]
and hence by \eqref{S6L2E3} again we have
\begin{equation}\label{S6L2E4}
     (N-2)\rho^{N-2}u^*_\mu(\rho)+r^{N-1}u^{*\prime}_\mu(r)+\int_r^\rho s^{N-1}H(s)\,ds+\rho^{N-2}\int_\rho^\infty sH(s)\,ds=0
\end{equation}
     for all $r<\rho$.
Letting $\rho\to\infty$ in \eqref{S6L2E4}, we obtain
     \[
     \lim_{\rho\to\infty}\rho^{N-2}u^*_\mu(\rho)=\frac{1}{N-2}\left(-r^{N-1}u^*_\mu(r)-\int_r^\infty s^{N-1}H(s)\,ds\right),
     \]
     which shows the existence of $\eta(\mu)$ and \eqref{u'inftyformula}.
The identity \eqref{uinftyformula} is obtained by integrating \eqref{u'inftyformula} over $(r,\infty)$.

     We finally observe from \eqref{u'inftyformula} and \eqref{uinftyformula} that
     \[
     (r^{N-2}u^*_\mu(r))'=r^{N-2}u^{*\prime}_\mu(r)+(N-2)r^{N-3}u^*_\mu(r)=r^{N-3}\int_r^\infty tH(t)\,dt>0,
     \]
and hence $r^{N-2}u^*_{\mu}(r)$ is strictly increasing in $r$.
\end{proof}

\begin{lemma}\label{vbarhittingzero}
  Assume the same conditions as in Theorem~\ref{S1T3}. There exists a unique solution $\overline{v}(\,\cdot\,,1)\in C^2(0,\infty)$ to the problem
  \begin{equation}\label{vbar1}
\bv''+\frac{N-1}{r}\bv'+k_\infty r^{\beta}\max\{\bv^p,0\}=0\ \textrm{for $r>0$},\ \lim_{r\to\infty}r^{N-2}\bv (r)=1.
  \end{equation}
  Furthermore, we have
  \[
  0<\overline{r}:=\sup\{r:\overline{v}(r,1)=0\}<\infty.
  \]
\end{lemma}
  \begin{proof}
    To prove the existence of a solution to \eqref{vbar1}, we use the Kelvin transform $\tilde{v}(r')=r^{N-2}\bv(r,1)$ and $r'=r^{-1}$.
    Then the problem \eqref{vbar1} is equivalent to
    \[
      \tilde{v}''+\frac{N-1}{r'}\tilde{v}'+k_\infty (r')^{\tilde{\beta}}\max\{\tilde{v}^p,0\}=0\ \textrm{for $r'>0$},\ \lim_{r'\to 0}\tilde{v}(r')=1,
    \]
    where $\tilde{\beta}=(N-2)(p-1)-4-\beta$.
Note that $p>p_S(\beta)$ is equivalent to $1<p<p_S(\tilde{\beta})$.
Hence, this problem has a unique solution $\tilde{v}$ and $\tilde{v}$ has the first zero $\tilde{r}$. Hence the problem \eqref{vbar1} has a unique solution $\overline{v}(\cdot,1)$ and $\overline{v}(\cdot,1)$ has the last zero $\overline{r}=\tilde{r}^{-1}$.
  \end{proof}
Let
  \[
  \tilde{\theta}:=\frac{2+\beta}{p-1}\ \ \text{and}\ \ \tilde{c}:=N-2-\tilde{\theta}.
  \]
Let $\overline{v}(\,\cdot\,,1)$ be defined in Lemma~\ref{vbarhittingzero}.
We define
  \[
\overline{v}(r,\eta):=\eta^{-\tilde{\theta}/\tilde{c}}\bar{v}(\eta^{-1/\tilde{c}}r,1)
  \]
  for $r>0$, $\eta>0$. Then $\overline{v}(\cdot,\eta)$ is a solution to the problem
  \begin{equation*}
  \bv''(r,\eta)+\frac{N-1}{r}\bv'(r,\eta)+k_\infty r^{\beta}\max\{\bv(r,\eta)^p,0\}=0\ \textrm{for $r>0$},\ \lim_{r\to\infty}r^{N-2}\bv(r,\eta)=\eta.
  \end{equation*}
  Furthermore, $\overline{v}(r,\eta)$ has the last zero $r=\eta^{1/\tilde{c}} \overline{r}$. We also have that
  \begin{equation}\label{vbarformula}
    \overline{v}(r,\eta)=r^{2-N}\eta-\frac{1}{N-2}\int_r^\infty (r^{2-N}-t^{2-N})t^{N-1}k_\infty t^{\beta}\overline{v}(t,\eta)^p\,dt.
  \end{equation}

\begin{lemma}\label{SFsolbound}
  Assume the same conditions as in Theorem~\ref{S1T3}. There exist constants $0<\underline{\eta}<\overline{\eta}$ and $R>0$ such that
  \[
  \underline{\eta}\le r^{N-2}u^*_\mu(r)\le\overline{\eta}
  \]
  for all $\mu\in\calF$ and $r>R$.
\end{lemma}
  \begin{proof}
    We first prove the existence of $R>0$ and $\overline{\eta}>0$ such that
    \begin{equation}\label{SFsolbound1}
      r^{N-2}u^*_\mu(r)\le\overline{\eta}\ \textrm{for all $\mu\in\calF$ and $r>R$}.
    \end{equation}
    Since $r^{N-2}u^*_\mu(r)$ is strictly increasing, it suffices to prove the uniform boundedness of $\eta(\mu)$ for $\mu\in\calF$.
    
    Assume on the contrary that there is a sequence $\{\mu_n\}_{n=1}^\infty\subset\calF$ such that $\lim\limits_{n\to\infty}\eta(\mu_n)=\infty$. By subtracting \eqref{vbarformula} with $\eta=\eta(\mu_n)$ from \eqref{uinftyformula} in Lemma~\ref{SFsollimit}, we see that
    \begin{align*}
    u^*_{\mu_n}(r)-\overline{v}(r,\eta(\mu_n))&=-\frac{1}{N-2}\int_r^\infty(r^{2-N}-t^{2-N})t^{N-1}(K(t)u^*_{\mu_n}(t)^p-k_\infty t^{\beta}\overline{v}(t,\eta(\mu_n))^p)\,dt\\
    &\quad-\frac{1}{N-2}\int_r^\infty (r^{2-N}-t^{2-N})t^{N-1}\mu_n f(t)\,dt.
    \end{align*}
This together with $u^*_{\mu_n}(r),\overline{v}(r,\eta(\mu_n))\le \eta(\mu_n)r^{2-N}$, Lemma~\ref{mubound} and the mean value theorem implies that
    \begin{align*}
      r^{N-2}|u^*_{\mu_n}(r)-\overline{v}(r,\eta(\mu_n))|
      &\le\frac{1}{N-2}\int_r^\infty k_\infty t^{N-1+\beta}|u^*_{\mu_n}(t)^p-\overline{v}(t,\eta(\mu_n))^p|\,dt\\
      &\qquad +\frac{1}{N-2}\int_r^\infty t^{N-1+\beta}|t^{-\beta} K(t)-k_\infty|u^*_{\mu_n}(t)^p\,dt\\
      &\qquad\qquad+\frac{1}{N-2}\int_r^\infty t^{N-1}\mu^* f(t)\,dt\\
      &\le\frac{1}{N-2}\int_r^\infty pk_\infty\eta(\mu_n)^{p-1} t^{N-1+\beta-(N-2)(p-1)}|u^*_{\mu_n}(t)-\overline{v}(t,\eta(\mu_n))|\,dt\\
      &\qquad+\tilde{\Phi}(r,\eta),
    \end{align*}
    where
    \[
    \tilde{\Phi}(r,\eta)=\frac{1}{N-2}\int_r^\infty t^{N-1+\beta}|t^{-\beta} K(t)-k_\infty|u^*_{\mu_n}(t)^p\,dt+\frac{1}{N-2}\int_r^\infty t^{N-1}\mu^* f(t)\,dt.
    \]
Let $\Psi(r):=r^{N-2}\left|u^*_{\mu_n}(r)-\overline{v}(r,\eta(\mu_n))\right|$.
Since $N-1+\beta-(N-2)(p-1)=N-2-(p-1)\tilde{c}-1$, we have
$$
\Psi(r)\le C\int_r^{\infty}\eta(\mu_n)^{p-1}t^{-(p-1)\tilde{c}-1}\Psi(t)dt+\tilde{\Phi}(r,\eta(\mu_n)).
$$
We change the variable $s:=t^{-1}$. Then
$$
\Psi(r)\le C\int_0^{r^{-1}}\eta(\mu_n)^{p-1}s^{(p-1)\tilde{c}-1}\Psi(s^{-1})ds+\tilde{\Phi}(r,\eta(\mu_n)).
$$
Let $\tau:=r^{-1}$ and $\tilde{\Psi}(\tau):=\Psi(r)$. Then
$$
\tilde{\Psi}(\tau)\le C\int_0^{\tau}\eta(\mu_n)^{p-1}s^{(p-1)\tilde{c}-1}\tilde{\Psi}(s)ds+\tilde{\Phi}(\tau^{-1},\eta(\mu_n)).
$$
By the Gronwall inequality, there is $C_2>0$ such that
$$
\tilde{\Psi}(\tau)\le\exp( C_2\eta(\mu_n)^{p-1}\tau^{(p-1)\tilde{c}})\tilde{\Phi}(\tau^{-1},\eta(\mu_n)),
$$
and hence
\begin{equation}\label{inftydif}
r^{N-2}|u^*_{\mu_n}(r)-\overline{v}(r,\eta(\mu_n))|\le e^{C_2\eta(\mu_n)^{p-1}r^{-(p-1)\tilde{c}}}\tilde{\Phi}(r,\eta(\mu_n))
    \end{equation}
    for all $r>0$. Furthermore, it follows from \eqref{Kinfty} and \eqref{finfty} that
    \begin{equation}\label{inftyPhi}
      \begin{aligned}
      \tilde{\Phi}(r,\eta)&=o(1)\eta(\mu_n)^p\int_r^\infty t^{N-1+\beta-(N-2)p}\,dt+o(1)\\
      &=\eta(\mu_n)^pr^{-(p-1)\tilde{c}}o(1)+o(1)\ \textrm{as $r\to\infty$ uniformly in $n\in\Z_{\ge 1}$.}
      \end{aligned}
    \end{equation}
    Let $\sigma_1>0$ be such that $\overline{v}(\sigma_1,1)<0$. Then it follows from \eqref{inftydif} and \eqref{inftyPhi} that
    \begin{align*}
      (\sigma_1\eta(\mu_n)^{1/\tc})^{N-2}|u^*_{\mu_n}(\sigma_1\eta(\mu_n)^{1/\tc})-
      &\overline{v}(\sigma_1\eta(\mu_n)^{1/\tc},\eta(\mu_n))|\\
      &\le C\tilde{\Phi}(\sigma_1\eta(\mu_n)^{1/\tc},\eta(\mu_n))\\
      &=\eta(\mu_n)^p\cdot {\sigma_1}^{-(p-1)\tc}\eta(\mu_n)^{-(p-1)}o(1)+o(1)\\
      &=\eta(\mu_n)o(1)\ \textrm{as $n\to\infty$}.
    \end{align*}
    On the other hand, we have
    \[
    (\sigma_1\eta(\mu_n)^{1/\tc})^{N-2}\overline{v}(\sigma_1\eta(\mu_n)^{1/\tc},\eta(\mu_n))=\sigma_1^{N-2}\eta(\mu_n)\overline{v}(\sigma_1,1).
    \]
    We deduce that
    \[
   (\sigma_1\eta(\mu_n)^{1/\tc})^{N-2}u^*_{\mu_n}(\sigma_1\eta(\mu_n)^{1/\tc})=\eta(\mu_n)(\sigma_1^{N-2}\overline{v}(\sigma_1,1)+o(1))\ \textrm{as $n\to\infty$,} 
    \]
    which contradicts $u^*_{\mu_n}(r)>0$.
Thus, there exist $R>0$ and  $\overline{\eta}>0$ such that \eqref{SFsolbound1} holds.

We next prove the existence of  $\underline{\eta}>0$ such that
    \begin{equation}\label{eq:SFsolbound2}
            r^{N-2}u^*_\mu(r)\ge\underline{\eta}\ \textrm{for all $\mu\in\calF$ and $r>R$}.
    \end{equation}
     Since $\calF\subset[0,\mu^*]$, by \eqref{eq:u*bddbelow} we see that there is a constant $\rho>0$ such that
    \[
    u^*(r)\ge\frac{\gamma}{2}r^{-\theta}\ \textrm{for all}\ r\in(0,\rho]\ \textrm{and}\ \mu\in\calF.
    \]
    Since $r^{N-2}u_{\mu}^*(r)$ is strictly increasing, this implies that
    \[
    r^{N-2}u_\mu^*(r)\ge \frac{\gamma}{2}\rho^{N-2-\theta}\ \textrm{for all}\ r\ge\rho\ \textrm{and}\ \mu\in\calF.
    \]
Thus, there exists $\underline{\eta}>0$ such that \eqref{eq:SFsolbound2} holds.
The proof is complete.
  \end{proof}
  
\begin{lemma}\label{SFsolconstraint}
  Assume the same conditions as in Theorem~\ref{S1T3}. Let $R>0$, $\underline{\eta}$, and $\overline{\eta}$ be as in Lemma~\ref{SFsolbound}. Then there exist $R_1>R$ and a real analytic function $\Xi$ on $(R_1^{2-N}\underline{\eta}/2,2R_1^{2-N}\overline{\eta})\times(-2\mu^*,2\mu^*)$ with the following properties:
  \begin{enumerate}
    \item $\mu^*\notin\calP_{R_1}$,
    \item we have
  \begin{equation}\label{Fconstraint}
  \calF=\{\mu\in\calP_{R_1}\cap[0,\infty): R_1^{N-2}u^*_\mu(R_1)\in(\underline{\eta}/2,2\overline{\eta}),u^{*\prime}_\mu(R_1)=\Xi(u^*_\mu(R_1),\mu)\}.
  \end{equation}
  \end{enumerate}
\end{lemma}
  \begin{proof}
    Since $u^*_{\mu^*}$ is non-increasing and becomes negative at some point in $(0,\infty)$ by Lemma~\ref{mubound}, there exsits $R_0>0$ such that $u^*_{\mu^*}\not\in\calP_{R_0}$.
In the below we choose $R_1>R_0$, and hence (1) in Lemma~\ref{SFsolconstraint} is satisfied.

We consider the integral equation
    \begin{equation}\label{vie}
    v(r)=r^{2-N}\eta-\mu F(r)-J[v](r)\ \textrm{for $r\in[R,\infty)$},\\
    \end{equation}
    where $\eta\in\R$, $\mu\in\R$ and
\begin{align*}
F(r)&:=\frac{1}{N-2}\int_r^\infty (r^{2-N}-t^{2-N})t^{N-1}f(t)\,dt,\\
J[v](r)&:=\frac{1}{N-2}\int_r^\infty (r^{2-N}-t^{2-N})t^{N-1}K(t)\max\{v(t),0\}^p\,dt.
\end{align*}
Note that if \eqref{vie} has a positive solution $v$, then $v$ satisfies \eqref{uinftyformula}, and hence $v$ satisfies the following:
    $$
    v''+\frac{N-1}{r}v'+K(r)\max\{v^p,0\}+\mu f(r)=0\ \ \text{for large $r>0$, and}
\lim_{r\to\infty}r^{-N+2}v(r)=\eta.
    $$
    
    We find a solution $v$ of \eqref{vie} in the space
    \[
    C_{2-N}[R,\infty):=\left\{h\in C[R,\infty):\sup_{r\ge R}r^{N-2}|h(r)|<\infty\right\}.
    \]
  Then $C_{2-N}[R,\infty)$ is a Banach space with the norm $\sup_{r\ge R}r^{N-2}|h(r)|$.
  
  Let $D:=[\underline{\eta}/4,4\overline{\eta}]\times[-2\mu^*,2\mu^*]$ and
    \begin{gather*}
    F_R:=\left\{h\in C_{2-N}[R,\infty):\frac{1}{8}\underline{\eta}\le r^{N-2}h(r)\le 8\overline{\eta}\ \ \textrm{for}\ \ r\ge R\right\},\\
    {F_R^\circ:=\left\{h\in C_{2-N}[R,\infty):\exists \varepsilon>0\ \left(\frac{1}{8}+\varepsilon\right)\underline{\eta}\le r^{N-2}h(r)<(8-\varepsilon)\overline{\eta}\ \ \textrm{for}\le \ r\ge R\right\}.}
    \end{gather*}
By \eqref{finfty} we first observe that
    \begin{equation}\label{F}
      \sup_{r\in[R,\infty)}r^{N-2}|F(r)|=o(1) \textrm{ as $R\to\infty$}.
    \end{equation}
Hereafter, $r$ is larger than $R$.
    For $v,v_1,v_2\in F_R$, we estimate
    \begin{align}
    \begin{split}
      r^{N-2}|J[v](r)|&\le C\int_r^\infty t^{N-1}K(t)\max\{v(t),0\}^p\,dt\\
      &\le C\overline{\eta}^p\int_r^\infty t^{N-1+\beta-(N-2)p}\,dt\\
      &\le Cr^{-(p-1)\tilde{c}}\overline{\eta}^p\\
      &\le CR^{-(p-1)\tilde{c}}\overline{\eta}^p,
    \end{split}\label{vest1}\\
    \begin{split}
      r^{N-2}|J[v_1](r)-J[v_2](r)|&\le C\int_r^\infty t^{N-1}K(t)\left|\max\{v_1(t),0\}^p-\max\{v_2(t),0\}^p\right| dt\\
      &\le C\overline{\eta}^{p-1}\sup_{r'\in[R,\infty)}(r')^{N-2}|v_1(r')-v_2(r')|\int_r^\infty t^{N-1+\beta-(N-2)p}\,dt\\
      &\le CR^{-(p-1)\tilde{c}}\overline{\eta}^{p-1}\sup_{r'\in[R,\infty)}(r')^{N-2}|v_1(r')-v_2(r')|.
    \end{split}\label{vest2}
    \end{align}
Taking $R_1>R_0$ sufficiently large, by \eqref{F}, \eqref{vest1}, and \eqref{vest2} we have that for each $(\eta,\mu)\in D$,
    \begin{gather}
\sup_{r\in[R_1,\infty)}r^{N-2}F(r)\le\frac{1}{16\mu^*}\underline{\eta},\qquad
\sup_{r\in[R_1,\infty)}r^{N-2}\left|J[v](r)\right|\le \frac{1}{8}\underline{\eta},\label{vie1}\\
\frac{1}{8}\underline{\eta}<r^{N-2}{(\eta r^{2-N}-\mu F(r)-J[v](r))}<8\overline{\eta},\label{vie2}\\
    \sup_{r\in[R_1,\infty)}r^{N-2}\left|J[v_1](r)-J[v_2](r)\right|\le\frac{1}{3}\sup_{r\in[R_1,\infty)}r^{N-2}|v_1(r)-v_2(r)|,\label{vie3}
    \end{gather}
    for all $v,v_1,v_2\in F_{R_1}$.
Because of \eqref{vie2} and \eqref{vie3}, the contraction mapping principle is applicable.
Hence, for all $(\eta,\mu)\in D$, there exists a unique solution $v_{\eta,\mu}\in F_{R_1}$ to the problem \eqref{vie}.
    
Next, we consider the mapping $\calM$ defined by
    \begin{gather*}
    \calM:D^\circ\times F_{R_1}^\circ\to{C_{2-N}[R_1,\infty)},\\
    \calM[\eta,\mu,v](r)=v(r)-\eta r^{2-N}+\mu F(r)+J[v](r).
    \end{gather*}
Since {$J$ is real analytic as a mapping $F^\circ_{R_1}\to C_{2-N}[R_1,\infty)$} {(see \textit{e.g.}, \cite{Z86} for detail)}, $\calM$ is also real analytic.
{
Indeed, for any $v\in F^\circ_{R_1}$ and $\varphi\in C_{2-N}[R_1,\infty)$ with $\sup_{r\ge R}r^{N-2}\varphi(r)$ small, we have that
\[
\psi(r):=\frac{\varphi(r)}{v(r)}\in L^\infty(R_1,\infty),
\]
with $\delta:=\|\psi\|_{L^\infty(R_1,\infty)}$ small. For any $n\in\N$, we see that
\begin{align*}
  J[v+\varphi](r)&=\frac{1}{N-2}\int_r^\infty(r^{2-N}-t^{2-N})t^{N-1}K(t)v(t)^p(1+\psi(t))^p\,dt\\
  &=\frac{1}{N-2}\int_r^\infty(r^{2-N}-t^{2-N})t^{N-1}K(t)v(t)^p\left(\sum_{k=0}^n\begin{pmatrix}
    p\\
    k
  \end{pmatrix}\psi(t)^k+O(\delta^{n+1})\right)\,dt\\
  &=\sum_{k=0}^n T_k[\varphi,\ldots,\varphi](r)+\int_r^\infty r^{2-N}t^{N-1+\beta-(N-2)p}O(\delta^{n+1})\,dt\\
  &=\sum_{k=0}^n T_k[\varphi,\ldots,\varphi](r)+O(\delta^{n+1})r^{2-N-(p-1)\tilde{c}}.
\end{align*}
where $\begin{pmatrix}
    p\\
    k
  \end{pmatrix}$ denotes the generalized binomial coefficient, and
  \[
  T_k[\varphi_1,\ldots,\varphi_k](r):=\frac{1}{N-2}\begin{pmatrix}
    p\\
    k
  \end{pmatrix}\int_r^\infty(r^{2-N}-t^{2-N})t^{N-1}K(t)v(t)^p\prod_{j=1}^k\frac{\varphi_j(t)}{v(t)}\,dt.
  \]
  By the same calculation as in \eqref{vest1}, $T_k$ is a continuous $k$-linear operator from $C_{2-N}[R_1,\infty)^k$ to $C_{2-N}[R_1,\infty)$. Hence
  \[
  \sup_{r\ge R_1}r^{N-2}\left|J[v+\varphi](r)-\sum_{k=0}^n T_k[\varphi,\ldots,\varphi]\right|\le CR_1^{-(p-1)\tilde{c}}\left(\frac{\sup_{r\ge R_1}r^{N-2}|\varphi(r)|}{\inf_{r\ge R_1}r^{N-2}v(r)}\right)^{n+1}
  \]
  implies the analyticity of $J$.
}

Furthermore, it follows from \eqref{vie3} that
     \[
     \sup_{r\in[R_1,\infty)}r^{N-2}|(\calM_v[\eta,\mu,v]h)(r)|\ge \frac{2}{3}\sup_{r\in[R_1,\infty)}r^{N-2}|h(r)|
     \]
    for all $(\eta,\mu,v)\in D^\circ\times F^\circ_{R_1}$ and $h\in C_{2-N}[R_1,\infty)$. By the implicit function theorem, $(\eta,\mu)\mapsto v_{\eta,\mu}$ is a real analytic map from $D^\circ$ to $F_{R_1}^\circ$. In particular, the function $V$ on $D^\circ$ defined by
    \[
    V(\eta,\mu):=R_1^{N-2}v_{\eta,\mu}(R_1)
    \]
is real analytic.
It follows from \eqref{vie1} that for all $(\eta,\mu)\in D$,
\begin{equation}\label{v-em}
\sup_{r\in[R_1,\infty)}r^{N-2}\mu\left|F(r)\right|
+\sup_{r\in[R_1,\infty)}r^{N-2}\left|J[v](r)\right|\\
\le \frac{\mu}{16\mu^*}\underline{\eta}+\frac{1}{8}\underline{\eta}
\le \frac{1}{4}\underline{\eta},
\end{equation}
where we used $\mu\le 2\mu^*$ and \eqref{vie1}.
Since $r^{N-2}v_{\eta,\mu}(r)=\eta-r^{N-2}\mu F(r)-r^{N-2}J[v_{\eta,\mu}](r)$, by \eqref{v-em} we have
\[
R_1^{N-2}v_{4\overline{\eta},\mu}(r)\ge 4\overline{\eta}-\frac{1}{4}\underline{\eta}>2\overline{\eta},\quad R_1^{N-2}v_{\underline{\eta}/4,\mu}(r)\le \frac{1}{4}\underline{\eta}+\frac{1}{4}\underline{\eta}=\frac{1}{2}\underline{\eta}
     \]
      for all $\mu\in(-2\mu^*,2\mu^*)$.
       It follows from the intermediate value theorem that
      \[
      \left(\underline{\eta}/2,2\overline{\eta}\right)\subset V\left(\left(\underline{\eta}/4,4\overline{\eta}\right)\times\{\mu\}\right)
      \]
      for all $\mu\in(-2\mu^*,2\mu^*)$.
      Furthermore, by \eqref{vie3}, for all $\mu\in(-2\mu^*,2\mu^*)$ and $\eta_1,\eta_2\in(\underline{\eta}/4,4\overline{\eta})$, we have
    \begin{align*}
    \sup_{r\in[R_1,\infty)}r^{N-2}|v_{\eta_1,\mu}(r)-v_{\eta_2,\mu}(r)|&=\sup_{r\in[R_1,\infty)}\left|\eta_1-\eta_2+r^{N-2}(J[v_{\eta_1,\mu}](r)-J[v_{\eta_2,\mu}](r))\right|\\
    &\le |\eta_1-\eta_2|+\frac{1}{3}\sup_{r\in[R_1,\infty)}r^{N-2}\left|v_{\eta_1,\mu}(r)-v_{\eta_2,\mu}(r)\right|,
    \end{align*}
    and thus
    \[
    \sup_{r\in[R_1,\infty)}r^{N-2}|v_{\eta_1,\mu}(r)-v_{\eta_2,\mu}(r)|\le\frac{3}{2}|\eta_1-\eta_2|.
    \]
    This implies that
    \begin{align*}
      R_1^{N-2}|v_{\eta_1,\mu}(R_1)-v_{\eta_2,\mu}(R_1)|&\ge |\eta_1-\eta_2|-R_1^{N-2}\left|(J[v_{\eta_1,\mu}](r)-J[v_{\eta_2,\mu}](r))\right|\\
      &\ge |\eta_1-\eta_2|-\frac{1}{3}\sup_{r\in[R_1,\infty)}r^{N-2}|v_{\eta_1,\mu}(r)-v_{\eta_2,\mu}(r)|\\
      &\ge \frac{1}{2}|\eta_1-\eta_2|,
    \end{align*}
    and thus
    \[
    V_\eta(\eta,\mu)\ge \frac{1}{2}
    \]
    for all $(\eta,\mu)\in D^\circ$. By the implicit function theorem, there is a real analytic function $\calH$ on $D':=(\underline{\eta}/2,2\overline{\eta})\times(-2\mu^*,2\mu^*)$ such that
    \begin{equation}\label{q}
      (\eta,\mu)\in D^\circ,R_1^{N-2}v_{\eta,\mu}(R_1)=\xi\iff \eta=\calH(\xi,\mu)
    \end{equation}
    for all $(\xi,\mu)\in D'$. Finally, we define $\Xi:(R_1^{2-N}\underline{\eta}/2,2R_1^{2-N}\overline{\eta})\times(-2\mu^*,2\mu^*)\to\R$ by
    \[
    \Xi(u,\mu)=R_1^{1-N}\left\{-(N-2)\calH(R_1^{N-2}u,\mu)+\int_{R_1}^\infty s^{N-1}\left(K(s){v^{*}_{\calH(R_1^{N-2}u,\mu),\mu}(s)^p}+\mu f(s)\right)\,ds\right\}.
    \]
    Note that the definition of $\Xi$ comes from \eqref{u'inftyformula} and that $\Xi$ is analytic.
We prove \eqref{Fconstraint}, especially
\begin{equation}\label{S6L5E1}
    u^{*\prime}_{\mu}(R_1)=\Xi(u^*_\mu(R_1),\mu).
\end{equation}

Let $\mu\in\calF$. It follows from Lemmas~\ref{mubound}~and~\ref{SFsolbound} that $(R_1^{N-2}u^*_\mu(R_1),\mu)\in D^\circ$. Furthermore, by \eqref{uinftyformula} and Lemma~\ref{SFsolbound}, we have $u^*_\mu=v_{\eta(\mu),\mu}$ on $[R_1,\infty)$. This together with \eqref{q} implies that
    \[
    \calH(R_1^{N-2}u^*_\mu(R_1),\mu)=\eta(\mu).
    \]
    Hence, by \eqref{u'inftyformula}, we deduce \eqref{S6L5E1}.
    Conversely, we assume that $\mu\in\calP_{R_1}\cap[0,\infty)$ satisfies $R_1^{N-2}u^*_\mu(R_1)\in(\underline{\eta}/2,2\overline{\eta})$ and \eqref{S6L5E1}.
    Let $v:=v_{\calH(R_1^{N-2}u^*_\mu(R_1),\mu),\mu}$.
Then a direct calculation shows that
    \begin{gather*}
    v''+\frac{N-1}{r}v'+K(r)v^p_\mu+\mu f(r)=0\ \textrm{for $r>R_1$},\\
    v(R_1)=u^*_\mu(R_1),\ v'(R_1)=\Xi(u^*_\mu(R_1),\mu).
    \end{gather*}
By defining as $ u^*_\mu \equiv v$ on $[R_1,\infty)$, we can extend $u^*_\mu$ into a positive singular fast-decay solution to \eqref{ODE} on $(0,\infty)$.
Thus, $\mu\in\calF$.

In summary, $\mu^*$ satisfies (1) and $\Xi$ constructed here satisfies (2).
The proof is complete.
  \end{proof}

\begin{lemma}\label{SFfinite}
  Assume the same conditions as in Theorem~\ref{S1T3}. Then $\calF$ is a finite set.
\end{lemma}
\begin{proof}
Assume on the contrary that $\calF$ is an infinite set. We define
  \[
  \calU:=\{\mu\in\calP_{R_1}: R_1^{N-2}u^*_\mu(R_1)\in(\underline{\eta}/2,2\overline{\eta})\},
  \]
  where $R_1$ is as in Lemma~\ref{SFsolconstraint}. It is clear that $\calU$ is an open subset of $\R$. It also follows from \eqref{SFsolbound} that $\calF\subset\calU$.
  
  Since $\calF$ is bounded by \eqref{SFsolbound}, $\calF$ has an accumulation point $\mu_*\in\R$. Furthermore, since $R_1^{N-2}u^*_{\mu_*}(R_1)\in[\underline{\eta},\overline{\eta}]$ by $\mu_*\in\overline{\calF}$ and $\eqref{SFsolbound}$, we have $\mu_*\in\calU$. Let $I=(\lambda_1,\lambda_2)$ be the connected component of $\calU$ that includes $\mu_*$. It follows from Theorem~\ref{singexists} and Lemma~\ref{SFsolconstraint} that
  \[
  \calF\cap I=\{\mu\in I\cap[0,\infty): R_1^{N-2}u^*_\mu(R_1)\in(\underline{\eta}/2,2\overline{\eta}),H(\mu)=0\},
  \]
where
  \[
  H(\mu):=u^{*\prime}_\mu(R_1)-\Xi(u^*_\mu(R_1),\mu).
  \]
  By Theorem~\ref{singexists} and Lemma~\ref{SFsolconstraint}, $H$ is a real analytic function on $I$. Since $\mu^*\in I$ is an accumulation point of $\calF\cap I$, by the identity theorem, we have
  \[
  I\cap[0,\infty)\subset\calF.
  \]
 Furthermore, by the choice of $R_1$, we have $\mu_*\le\mu^*<\infty$. It follows from $\lambda_2\in\partial(I\cap[0,\infty))\subset\overline{\calF}$ that
  \[
  R_1^{N-2}u^*_{\lambda_2}(R_1)\in[\underline{\eta},\overline{\eta}].
  \]
  By the definition of $\calU$, it holds that $\lambda_2\in\calU$, which contradicts the fact that $(\lambda_1,\lambda_2)$ is a connected component of $\mathcal{U}$.
\end{proof}

  We next investigate asymptotic behavior of slow decay solutions. We define $\tilde{w}(t):=e^{\tilde{\theta}t}u(e^t)$. We see that $\tilde{w}$ is a solution to equation
  \begin{equation}\label{tweq}
    \tilde{w}''+\tilde{a}\tilde{w}'-\tilde{A}^{p-1}\tilde{w}+\tilde{L}(t)\tilde{w}^p+\mu\tilde{g}=0\ \textrm{for $t>0$},
  \end{equation}
  where
  \[
  \tilde{a}:=N-2-2\tilde{\theta},\ \ \tilde{A}^{p-1}:=\tilde{\theta}\tilde{c},
  \ \ \tilde{L}(t):=e^{\beta t}K(e^t),\ \ \tilde{g}(t):=e^{(2+\tilde{\theta})t}f(e^t).
  \]
  We also define
  \[
  \tilde{\gamma}:=k_\infty^{\frac{1}{p-1}} \tilde{A}.
  \]
  
  \begin{lemma}\label{S6L7-}
  For any $\varepsilon>0$, there are $\varepsilon'>0$ and $\tau\in\R$ such that if
      \begin{equation}\label{Sdecayrate0-1}
       |\tilde{w}(t)-\tilde{\gamma}|<\varepsilon',\ |\tilde{w}'(t)|<\varepsilon',
      \end{equation}
      for some $t>\tau$, then
      \begin{equation}\label{Sdecayrate0-2}
      |\tilde{w}(t')-\tilde{\gamma}|<\varepsilon,\ |\tilde{w}'(t')|<\varepsilon,
      \end{equation}
      for all $t'>t$.
  \end{lemma}
\begin{proof}
The lemma can be prove in a similar argument to that of Lemma~~\ref{uapprox2}.
We omit the proof.
\end{proof}

  \begin{lemma}\label{SDasymp}
Let $N\ge 3$ and $p>p_S(\beta)$.
Assume that \eqref{Kinfty} and \eqref{finfty2} hold.
If $\tilde{w}$ is a solution to \eqref{tweq} satisfying,
    \[
    0<\limsup_{t\to\infty}\tilde{w}(t)<\infty,
    \]
    then
  \begin{equation}\label{SDecayrate}
\lim_{t\to\infty}\tilde{w}(t)=\tilde{\gamma}\ \ \text{and}\ \ 
\lim_{t\to\infty}\tilde{w}'(t)=0.
  \end{equation}
  \end{lemma}

\begin{proof}
    We prove that
    \begin{equation}\label{SDecayrate1}
      \limsup_{t\to\infty}\tilde{w}(t)=\liminf_{t\to\infty}\tilde{w}(t).
    \end{equation}
    Assume on the contrary that $\liminf\limits_{t\to\infty}\tilde{w}(t)<\limsup\limits_{t\to\infty}\tilde{w}(t)$. Then it follows that 
    \begin{equation}\label{Sdecayrate1-1}
      \liminf_{t\to\infty}\tilde{w}(t)\neq\tilde{\gamma}\ \ \text{and}\ \ 
      \limsup\limits_{t\to\infty}\tilde{w}(t)\neq\tilde{\gamma}.
    \end{equation}
    Indeed, if one of $\liminf\limits_{t\to\infty}\tilde{w}(t)$ and $\limsup\limits_{t\to\infty}\tilde{w}(t)$ equals to $\tilde{\gamma}$, we find a sequence $t_1<t_2<\cdots\to\infty$ such that $\lim\limits_{n\to\infty}\tilde{w}(t_n)=\tilde{\gamma}$ and $\tilde{w}'(t_n)=0$ for all $n\in\Z_{\ge 1}$.
This together with Lemma~\ref{S6L7-} implies that $\lim\limits_{t\to\infty}\tilde{w}(t_n)=\tilde{\gamma}$, which contradicts the assumption.

    We next prove that
$$
      \liminf_{t\to\infty}\tilde{w}(t)<\tilde{\gamma}<\limsup_{t\to\infty}\tilde{w}(t).
$$
    To see that $\gamma^*:=\limsup\limits_{t\to\infty}\tilde{w}(t)>\tilde{\gamma}$, we find a sequence ${\overline{t}_1}<{\overline{t}_2}<\ldots\to\infty$ such that
    \[
    \lim_{n\to\infty}\tilde{w}(\overline{t}_n)=\gamma^*,\ \tilde{w}'({\overline{t}_n})=0,\ \tilde{w}''({\overline{t}_n})\le 0,\ \textrm{for all $n\in\Z_{\ge 1}$}.
    \]
    Putting $t=\overline{t}_n$ in \eqref{tweq} and sending $n\to\infty$, by \eqref{finfty2} we see that
    \[
    -A^{p-1}\gamma^*+k_\infty\gamma^{*p}\ge 0,
    \]
    which together with $\gamma^*>0$ implies that $\gamma^*\ge\tilde{\gamma}$. Combining this with \eqref{Sdecayrate1-1}, we obtain $\gamma^*>\tilde{\gamma}$. Similarly, it follows that $\gamma_*:=\liminf\limits_{t\to\infty}\tilde{w}(t)<\tilde{\gamma}$.

    Let $\varepsilon>0$ be such that $\gamma^*>\tilde{\gamma}+\varepsilon$, $\gamma_*<\tilde{\gamma}-\varepsilon$. Let $\varepsilon'>0$ and $\tau>0$ be as in \eqref{Sdecayrate0-1}. Then we have
    \begin{equation}\label{Sdecayrate2-1}
      t>\tau,\ |\tilde{w}(t)-\tilde{\gamma}|<\varepsilon'\implies |\tilde{w}'(t)|\ge\varepsilon',
    \end{equation}
    since \eqref{Sdecayrate0-2} immediately contradicts the choice of $\varepsilon$ if $|\tilde{w}(t)-\tilde{\gamma}|<\varepsilon'$ and $|\tilde{w}'(t)|<\varepsilon'$ for some $t>\tau$. Furthermore, if $\tilde{w}'(t)=0$, then
    \[
    \tilde{w}''(t)=A^{p-1}\tilde{w}(t)-k_\infty\tilde{w}(t)^p-(L(t)-k_\infty)\tilde{w}(t)^p{-\mu}\tilde{g}(t).
    \]
    This together with \eqref{Kinfty} and \eqref{finfty2} implies the existence of $\tau'>\tau$, $\delta_0>0$, and $\varepsilon_0>0$ such that
    \begin{gather}
      \tilde{w}(t)\ge\tilde{\gamma}+\varepsilon',\ \tilde{w}'(t)=0\implies \tilde{w}''(t)<-\delta_0,\label{Sdecayrate2-1.1}\\
      \varepsilon_0\le\tilde{w}(t)\le\tilde{\gamma}-\varepsilon',\ \tilde{w}'(t)=0\implies \tilde{w}''(t)>\delta_0,\label{Sdecayrate2-1.2}
    \end{gather}
  for $t>\tau'$. Furthermore, $\varepsilon_0$ can be assumed to be arbitrarily small by replacing $\delta_0>0$ by a sufficiently small number. 
By \eqref{Sdecayrate2-1} and \eqref{Sdecayrate2-1.1} we see that
\begin{equation}\label{Sdecayrate2-1.3}
      \tilde{w}'(t)=0,\ \tilde{w}''(t)\ge 0\implies \tilde{w}(t)\le\tilde{\gamma}-\varepsilon'
\end{equation}
for $t>\tau'$.
By \eqref{Sdecayrate2-1} and \eqref{Sdecayrate2-1.2} we see that
\begin{equation}\label{Sdecayrate2-1.4}
      \tilde{w}(t)\ge\varepsilon_0,\ \tilde{w}'(t)=0,\ \tilde{w}''(t)\le 0\implies \tilde{w}(t)\ge\tilde{\gamma}+\varepsilon'
\end{equation}
for $t>\tau'$. 

    We now consider the quantity
    \[
    E(t):=\frac{1}{2}\tilde{w}'(t)^2-\frac{\tilde{A}^{p-1}}{2}\tilde{w}(t)^2+\frac{k_\infty}{p+1}\tilde{w}(t)^{p+1}.
    \]
    We see that
    \[
    E'(t)=-\tilde{a}\tilde{w}'(t)^2+\tilde{w}'(t)((k_\infty-L(t)){\tilde{w}(t)}^{p-1}-{\mu}\tilde{g}(t)).
    \]
    Let $\tau<\underline{t}<\overline{t}$ be such that $\tilde{w}$ is nondecreasing on $(\underline{t},\overline{t})$ and $\tilde{w}(\underline{t})\le\tilde{\gamma}-{\varepsilon'}$, $\tilde{w}(\overline{t})\ge\tilde{\gamma}+{\varepsilon'}$.
    Let $\underline{t}\le\underline{t}'<\overline{t}'\le\overline{t}$ be such that $\tilde{w}({\underline{t}'})=\tilde{\gamma}-{\varepsilon'}$, $\tilde{w}({\overline{t}'})=\tilde{\gamma}+{\varepsilon'}$. It follows from \eqref{Sdecayrate2-1} that $\tilde{w}'(t)\ge\varepsilon'$ for $t\in[\underline{t}',\overline{t}']$, and thus
    \[
    \int_{\underline{t}}^{\overline{t}}\tilde{w}'(t)^2\,dt\ge\int_{\underline{t}'}^{\overline{t}'}\varepsilon'\tilde{w}'(t)\,dt
=\e'\left(\tilde{w}(\overline{t}')-\tilde{w}(\underline{t}')\right)
=2\varepsilon'^2.
    \]
    This together with \eqref{Kinfty} and \eqref{finfty} implies the existence of $\tau''>\tau'$ such that
    \begin{equation}\label{Sdecayrate2-2}
    E(\overline{t})-E(\underline{t})=\int_{\underline{t}}^{\overline{t}}E'(t)\,dt\le -\tilde{a}\varepsilon'^2
    \end{equation}
    if $\underline{t}>\tau''$ in addition.
    Similarly, we can assume that if $\underline{t}>\tau''$, $\tilde{w}$ is nonincreasing on $(\underline{t},\overline{t})$, $\tilde{w}(\underline{t})\ge\tilde{\gamma}+{\varepsilon'}$, and $\tilde{w}(\overline{t})\le\tilde{\gamma}-{\varepsilon'}$, then \eqref{Sdecayrate2-2} holds.
    
    Let $\varepsilon_0\in(0,\tilde{\gamma}-\varepsilon')$ be such that
    \begin{equation}\label{S6L7E1}
    w\ge 0,\ -\frac{\tilde{A}^{p-1}}{2}w^2+\frac{k_\infty}{p+1}w^{p+1}\le -2\tilde{a}\varepsilon'^2\implies w\ge\varepsilon_0.
    \end{equation}
    Let $\tau_0$ and $\tau_1$ be local maximum points of $\tilde{w}$ such that $\tau_1>\tau_0>\tau''$, $\tilde{w}(\tau_0)>\tilde{\gamma}+\varepsilon$ and $\tilde{w}(\tau_1)>\tilde{\gamma}+\varepsilon$. By \eqref{Sdecayrate2-1.1}, we have $\tilde{w}''(\tau_1)<0$. This implies that
    \[
    \tau_0':=\sup\{t<\tau_1:\tilde{w}'(t)=0\}<\tau_1,\ \tau_1':=\inf\{t>\tau_1:\tilde{w}'(t)=0\}>\tau_1.
    \]
    Furthermore, we have $\tilde{w}''(\tau_0'),\tilde{w}''(\tau_1')\ge 0$. This together with \eqref{Sdecayrate2-1.3} implies that $\tilde{w}(\tau_0'),\tilde{w}(\tau_1')\le\tilde{\gamma}-\varepsilon$.
    Using \eqref{Sdecayrate2-2} twice, we have
    \[
    E(\tau_1')-E(\tau_0')
=\left(E(\tau_1')-E(\tau_1)\right)+\left(E(\tau_1)-E(\tau_0')\right)
    \le -2\tilde{a}\varepsilon'^2.
    \]
    Since
    \[
    E(\tau_0')=-\frac{\tilde{A}^{p-1}}{2}\tilde{w}(\tau_0')^{2}+\frac{k_\infty}{p+1}\tilde{w}(\tau_0')^{p+1}<-\frac{1}{2}\tilde{w}(\tau_0')^2\left(\tilde{A}^{p-1}-k_\infty\tilde{w}(t_0')^{p-1}\right)<0
    \]
    by $0<\tilde{w}(\tau_0')<\tilde{\gamma}$, we obtain $E(\tau_1')<-2\tilde{a}\varepsilon'^2$. This together with \eqref{S6L7E1} implies that $\tilde{w}(\tau_1')\ge\varepsilon_0$.
    It also follows from \eqref{Sdecayrate2-1.2} that $\tilde{w}''(\tau_1')>\delta_0$. In particular, we have
    \[
    \tau_2:=\inf\{t>\tau_1':\tilde{w}'(t)=0\}>\tau_1'.
    \]
    Also, since $\tilde{w}(\tau_2)>\tilde{w}(\tau_1')\ge\varepsilon_0$, it follows from \eqref{Sdecayrate2-1.4} that $\tilde{w}(\tau_2)\ge\tilde{\gamma}+\varepsilon'$.

    Repeating this argument, we obtain a sequence $\tau''<\tau_1<\tau_1'<\tau_2<\tau_2'<\cdots$ such that
    \begin{gather*}
      \tilde{w}(\tau_n)\ge\tilde{\gamma}+\varepsilon',\ \tilde{w}'(\tau_n)=0,\ \tilde{w}''(\tau_n)<-\delta_0,\\
      \varepsilon_0\le\tilde{w}(\tau_n')\le\tilde{\gamma}-\varepsilon',\ \tilde{w}'(\tau_n')=0,\ \tilde{w}''(\tau_n')>\delta_0,\\
      \tilde{w}'<0\ \textrm{on $(\tau_n,\tau_n')$},\ {\tilde{w}'>0}\ \textrm{on $(\tau_n',{\tau_{n+1}})$},
    \end{gather*}
    for all $n\in\Z_{\ge 1}$. By \eqref{Sdecayrate2-2}, we deduce that
    \[
    E(\tau_n')-E(\tau_n)\le-\tilde{a}\varepsilon'^2,\ E(\tau_{n+1})-E(\tau_n')\le-\tilde{a}\varepsilon'^2,
    \]
    for all $n\in\Z_{\ge 1}$. Consequently, it follows that
    \[
    E(\tau_n)\le E(\tau_1)-2(n-1)\tilde{a}\varepsilon'^2.
    \]
    On the other hand, it follows from the definition of $E(t)$ that
    \[
    E(\tau_n)\ge-\frac{\tilde{A}^{p-1}}{2}\tilde{\gamma}^2+\frac{k_\infty}{p+1}\tilde{\gamma}^{p+1},
    \]
    which is a contradiction. Thus \eqref{SDecayrate1} follows.

    We finally prove \eqref{SDecayrate}. Assume on the contrary that $\gamma^*=\gamma_*\neq\tilde{\gamma}$. In the case $\gamma^*>\tilde{\gamma}$, since
    \[
-\tilde{A}^{p-1}\tilde{w}(t)+\tilde{L}(t)\tilde{w}(t)^p+\mu\tilde{g}(t)\to
    -\tilde{A}^{p-1}\gamma^*+k_\infty\gamma^{*p}>0\ \textrm{as $t\to\infty$,}
    \]
    there is $\tau\in\R$ and $\delta>0$ such that
    \[
    \tilde{w}''(t)+\tilde{a}\tilde{w}'(t)\ge\delta\ \textrm{for all $t>\tau$.}
    \]
    By the Gronwall inequality, we obtain
    \[
    \tilde{w}'(t)\ge \frac{\delta}{\tilde{a}}+\left(\tilde{w}'(\tau)-\frac{\delta}{\tilde{a}}\right)e^{-\tilde{a}(t-\tau)}\ \textrm{for all $t>\tau$.}
    \]
    In particular, there is $\tau'>\tau$ such that
    \[
    \tilde{w}'(t)\ge \frac{\delta}{2\tilde{a}}\ \textrm{for all $t>\tau$,}
    \]
    which contradicts the boundedness of $\tilde{w}$. A similar contradiction occurs in the case $0<\gamma^*<\tilde{\gamma}$. Thus we have $\gamma^*=\gamma_*=\tilde{\gamma}$.
The assertion $\lim\limits_{t\to\infty}\tilde{w}'(t)=0$ is proved by a similar way to the proof of \eqref{Sdecayrate1-1}.
    \end{proof}
  
\begin{proof}[Proof of Theorem \ref{S1T3}]
  By Lemma~\ref{SFfinite}, it remains to prove that $\calS$ and $\calB$ are open subsets of $[0,\infty)$.

  We first prove the openness of $\calB$. Let $\mu_0\in\calB$. Then there is $R_0>0$ such that $u^*_{\mu_0}(R_0)<0$. By the continuous dependence, there is $\varepsilon>0$ such that $u^*_{\mu}(R_0)<0$ for all $\mu\in(\mu_0-\varepsilon,\mu_0+\varepsilon)$. Hence $\calB$ is open.
  
We finally prove the openness of $\calS$.
Since $\beta>-2$ and $p>p_S(\beta)$, we see that $(2+\beta)/(p-1)<N-2$, and hence a slow-decay solution always decays slower than a fast-decay solution.
Let $\mu_0\in\calS$ and $\tilde{w}^*_\mu(t):=e^{\tilde{\theta}t}u(e^t)$ for $\mu\ge 0$.
A similar argument as in the proof of \eqref{asy2} derives
  \[
  0<\limsup_{t\to\infty}\tilde{w}^*_{\mu_0}(t)<\infty.
  \]
  By Lemma~\ref{SDasymp}, we obtain
  \[
  \lim_{t\to\infty}\tilde{w}^*_{\mu_0}(t)=\tilde{\gamma},\ \lim_{t\to\infty}\tilde{w}^{*\prime}_{\mu_0}(t)=0.
  \]
  Let $\varepsilon'>0$ and $\tau\in\R$ as in \eqref{Sdecayrate0-1}, \eqref{Sdecayrate0-2}, with $\varepsilon=\tilde{\gamma}/2$. By the continuous dependence, there is $\delta>0$ such that 
  \[
  |\tilde{w}^*_\mu(\tau+1)-\tilde{\gamma}|<\varepsilon',\ |\tilde{w}^{*\prime}_\mu(\tau+1)|<\varepsilon',
  \]
  for all $\mu\in(\mu_0-\delta,\mu_0+\delta)\cap[0,\infty)$.
By Lemma~\ref{S6L7-}, this implies that
  \[
  |\tilde{w}^*_\mu(t)-\tilde{\gamma}|<\frac{\tilde{\gamma}}{2}\ \textrm {for all $t>\tau+1$,}
  \]
  and thus $(\mu_0-\delta,\mu_0+\delta)\cap[0,\infty)\subset\calS$. The proof of Theorem~\ref{S1T3} is complete.
\end{proof}
\bigskip

\noindent{\bf Acknowledgements}
SK was supported by Grant-in-Aid for JSPS Fellows Grant Number 23KJ0645, and FoPM, WINGS Program, the University of Tokyo.
YM was supported by JSPS KAKENHI Grant Number 24K00530.\\

\noindent{\bf Conflict of interest}\\
The authors have no relevant financial or non-financial interests to disclose.\\

\noindent{\bf Data Availability}\\
Data sharing is not applicable to this article as no new data were created or analyzed in this study.


\end{document}